\documentclass[12pt]{amsart}
\usepackage{amsmath}
\usepackage{amssymb}
\usepackage{amsthm}
\usepackage{bm}

\usepackage[dvips]{graphicx}
\usepackage{color}

\numberwithin{equation}{section}
\newtheorem{theorem}{Theorem}[section]
\newtheorem{lemma}[theorem]{Lemma}
\newtheorem{proposition}[theorem]{Proposition}
\newtheorem{corollary}[theorem]{Corollary}
\newtheorem{definition}[theorem]{Definition}

\theoremstyle{definition}
\newtheorem{remark}[theorem]{Remark}
\newtheorem{example}[theorem]{Example}

\newcommand{\R}{\mathbb{R}}
\newcommand{\N}{\mathbb{N}}
\newcommand{\Prime}{\mathbb{P}}
\newcommand{\Q}{\mathbb{Q}}
\newcommand{\pn}{\par\noindent}
\newcommand{\rd}{\mathbb{R}^d}

\newcommand{\vc}{{\Vec{c}}}
\newcommand{\vs}{{\Vec{s}}}
\newcommand{\vt}{{\Vec{t}}}

\newcommand{\vsig}{{\Vec{\sigma}}}

\allowdisplaybreaks[3]
\begin{document}
\title[Zeta distributions generated by multidimensional Euler products]{Zeta distributions generated by multidimensional polynomial Euler products with complex coefficients}
\author[T.~Nakamura]{Takashi Nakamura}
\address[T.~Nakamura]{Department of Liberal Arts, Faculty of Science and Technology, Tokyo University of Science, 2641 Yamazaki, Noda-shi, Chiba-ken, 278-8510, Japan}
\email{nakamuratakashi@rs.tus.ac.jp}
\urladdr{https://sites.google.com/site/takashinakamurazeta/}
\maketitle

\begin{abstract}
In the present paper, we treat multidimensional polynomial Euler products with complex coefficients on ${\mathbb{R}}^d$. We give necessary and sufficient conditions for the multidimensional polynomial Euler products to generate infinitely divisible, quasi-infinitely divisible but non-infinitely divisible or not even characteristic functions by using Baker's theorem. Moreover, we give many examples of zeta distributions on ${\mathbb{R}}^d$ generated by the multidimensional polynomial Euler products with complex coefficients. Finally, we consider applications to analytic number theory.
\end{abstract}

\tableofcontents 

\section{Introduction}
\subsection{Infinitely divisible distributions}

In probability theory, infinitely divisible distributions are one of the most significant class of distributions. For example, Normal, degenerate, Poisson and compound Poisson distributions are infinitely divisible. The definition of infinitely divisible distributions is as follows.

\begin{definition}[Infinitely divisible distribution, see {\cite[Definition 7.1]{S99}}]
A probability measure $\mu$ on $\rd$ is infinitely divisible if,
for any positive integer $n$, there is a probability measure $\mu_n$ on 
$\rd$ such that $\mu=\mu_n^{n*}$, where $\mu_n^{n*}$ is the $n$-fold convolution of $\mu_n$.
\end{definition}

Denote by $I(\rd)$ the class of infinitely divisible distributions on $\rd$. Let $\widehat{\mu}(\vt):=\int_{\rd}e^{{\rm i}\langle \vt,x\rangle}\mu (dx),\, \vt\in\rd,$ be the characteristic function of a distribution $\mu$, 
where $\langle\cdot ,\cdot\rangle$ is the inner product. We write $a\wedge b=\min \{a,b\}$ as usual. 
\begin{proposition}[L\'evy--Khintchine representation, see {\cite[Theorem 8.1]{S99}})]\label{pro:LK1}
$(i)$ If $\mu\in I(\rd)$, then it holds that
\begin{equation}\label{INF}
\widehat{\mu}(\vt) = \exp \left[ -\frac{1}{2} \langle \vt,A\vt \rangle + {\rm i} \langle \gamma,\vt \rangle
+ \int_{\rd} \left( e^{{\rm i} \langle \vt,x \rangle} -1 -
\frac{{\rm i} \langle \vt,x \rangle}{1+|x|^2} \right) \nu(dx) \right], \quad \vt \in \rd,
\end{equation}
where $\gamma \in\rd$, $A$ is a symmetric nonnegative-definite $d \times d$ matrix, and $\nu$ is a measure on $\rd$ which satisfies
\begin{equation}\label{lev}
\nu(\{0\}) = 0 \quad \mbox{and} \quad \int_{\rd} (|x|^{2} \wedge 1) \nu(dx) < \infty.
\end{equation}

$(ii)$ The representation of $\widehat{\mu}$ in $(i)$ by $A, \nu,$ and $\gamma$ is unique.

$(iii)$ Conversely, if a symmetric $d\times d$ matrix $A$ is nonnegative-definite, a measure $\nu$ fulfills $\eqref{lev}$, and $\gamma \in\rd$, then there exists an infinitely divisible distribution $\mu$ whose characteristic function is given by $\eqref{INF}$.
\end{proposition}

The measure $\nu$ and $(A, \nu ,\gamma)$ in (\ref{INF}) are called the L\'evy measure and the L\'evy--Khintchine triplet of $\mu\in I(\rd)$, respectively. When the L\'evy measure $\nu$ satisfies an additional condition, one has a simpler form of $\eqref{INF}$.

\begin{proposition}[see, {\cite[(8.7)]{S99}}]\label{pro:LK2}
In Proposition \ref{pro:LK1}, if the L\'evy measure $\nu$ in $\eqref{INF}$ also satisfies $\int_{|x|<1}|x|\nu(dx)<\infty$, then we can rewrite the representation $\eqref{INF}$ by
\begin{equation}\label{INF2}
\widehat{\mu}(\vt) = \exp \left[ -\frac{1}{2} \langle \vt,A\vt \rangle + {\rm i} \langle \gamma_0,\vt \rangle
+ \int_{\rd} \left( e^{{\rm i} \langle \vt,x \rangle} -1 \right) \nu(dx) \right], \qquad \vt \in \rd,
\end{equation}
where $\gamma_0 = \gamma - \int_{\rd}x (1+|x|^2)^{-1} \nu(dx)$.
\end{proposition}

As an example, we consider the L\'evy--Khintchine representation of a compound Poisson distribution $\mu_{{\rm CPo}}$. For some $c>0$ and distribution $\rho$ on $\rd$ with $\rho (\{0\})=0$, we have
\begin{equation}\label{CPcf}
\widehat\mu_{{\rm CPo}}(\vt) = \exp \left( c\left( \widehat{\rho} (\vt) -1 \right) \right), \qquad \vt \in \rd. 
\end{equation}
Note that the Poisson distribution is a special case when $d=1$ and $\rho =\delta_1$.

\subsection{Riemann zeta function and Euler Products}
Zeta functions play one of the key roles in number theory. In 1859, Riemann established a relation between zeros of the Riemann zeta function and the distribution of prime numbers. The definition of the Riemann zeta function is as follows.

\begin{definition}[Riemann zeta function, see {\cite[Section 11]{Apo}}]
Let $s = \sigma + {\rm i}t$. Then for $\sigma >1$, the Riemann zeta function is a given by
\begin{equation}\label{eq:eupro}
\zeta (s) := \sum_{n=1}^{\infty} \frac{1}{n^s} = \prod_p \Bigl( 1 - \frac{1}{p^s} \Bigr)^{-1} ,
\end{equation}
where the letter $p$ is a prime number, and the product of $\prod_p$ is taken over all primes.
\end{definition}
The infinite series is called the Dirichlet series and the infinite product is called the Euler product. The Dirichlet series and the Euler product of $\zeta (s)$ converge absolutely when $\sigma >1$ and uniformly in each compact subset of the half-plane $\sigma >1$. 

The Dirichlet $L$-function $L(s,\chi)$ attached to a Dirichlet character $\chi$ mod $q$ is given by
\begin{equation}
L (s,\chi) := \sum_{n=1}^{\infty} \frac{\chi (n)}{n^s} = \prod_p \Bigl( 1 - \frac{\chi(p)}{p^s} \Bigr)^{-1} ,
\qquad \sigma > 1.
\end{equation}
The Riemann zeta function $\zeta (s)$ can be regarded as the Dirichlet $L$-function to the principal character $\chi_0 \mod 1$. 

Moreover, let $K$ be a general number field and $\mathbb{Z}_K$ be its ring of integers. As one of a generalization of $\zeta (s)$, it is natural to define the following function
$$
\zeta_K (s) := \sum_{{\mathfrak{a}}} \frac{1}{{\mathcal{N}} ({\mathfrak{a}})^s}
= \prod_{{\mathfrak{p}}} \Bigl( 1 - \frac{1}{{\mathcal{N}} ({\mathfrak{p}})^s} \Bigr)^{-1}, \qquad  \sigma > 1,
$$
where ${\mathfrak{a}}$ runs through all integral ideals of $\mathbb{Z}_K$ and ${\mathfrak{p}}$ through all prime ideals of $\mathbb{Z}_K$ and ${\mathcal{N}}$ denotes the absolute norm. 
The function $\zeta_K (s)$ is called the Dedekind zeta function. Obviously, we have $\zeta_{\Q} (s) = \zeta (s)$. Let $K=\Q (\sqrt{D})$ be a quadratic field of discriminant $D$.
Then we have $\zeta_K (s) = \zeta (s) L(s,\chi_D)$, where $\chi_D$ is the Legendre--Kronecker character (see \cite[Proposition 10.5.5]{Cohen}). Furthermore, we have the following (see, \cite[Theorem 10.5.22]{Cohen}). 
Let $\Q_m$ be the $m$-th cyclotomic field. Then one has $\zeta_{\Q_m} (s) = \prod_{\chi \,\, {\rm{mod}} \,\, m} L(s,\chi_f)$, where $\chi_f$ is the primitive character associated with $\chi$. In particular, we have $\zeta_{\Q_m} (s) = \prod_{\chi \,\, {\rm{mod}} \,\, m} L(s,\chi)$ when $m$ is a prime power. 

These well-known functions above can be regarded as the prototype of zeta functions which have the Euler products. 
Many authors have introduced and investigated classes of zeta or {\textit{L}}-functions to find the essential properties satisfied by functions with the Euler products. For example, in \cite{Steuding1}, there are two classes of Dirichlet series satisfying some quite natural analytic axioms with several arithmetic conditions added. 

\subsection{Zeta distributions and quasi-infinite divisibility}

In probability theory, there is a class of distribution on $\R$ generated by $\zeta (s)$. First it appears in \cite{Khi} and we can also find it in \cite{GK68}. Put
\begin{equation*}
f_{\sigma}(t):= \frac{\zeta (\sigma +{\rm i}t)}{\zeta (\sigma)}, \qquad t \in \R,
\end{equation*}
then $f_{\sigma}(t)$ is a characteristic function (see \cite[p.~76]{GK68}).

\begin{definition}[Riemann zeta distribution]
A distribution $\mu_{\sigma}$ on $\R$ is a Riemann zeta distribution with parameter $\sigma >1$ if it has $f_{\sigma}(t)$ as its characteristic function.
\end{definition}

The Riemann zeta distribution is infinitely divisible and its L\'evy measures can be given of the form as in the following.

\begin{proposition}[see {\cite[p.~76]{GK68}}]\label{pro:RD}
Let $\mu_{\sigma}$ be a Riemann zeta distribution on $\R$ with characteristic function $f_{\sigma}(t)$.
Then, $\mu_{\sigma}$ is compound Poisson on $\R$ and
\begin{align*}
f_{\sigma}(t) = \exp \left[ \int_0^{\infty} \!\! \left(e^{-{\rm i}tx}-1\right) N_{\sigma}(dx) \right], 
\qquad N_{\sigma}(dx):= \sum_{p} \sum_{r=1}^{\infty} \frac{p^{-r\sigma}}{r} \delta_{r\log p}(dx),
\end{align*}
where $\delta_x$ is the delta measure at $x$.
\end{proposition}

\begin{remark}
It should be mentioned that the Riemann zeta distribution is defined only in the region of half-plane $\sigma >1$ since normalized functions $\zeta(\sigma+it) /\zeta(\sigma)$ can not be characteristic functions for any $1/2 \le \sigma \le 1$ (see \cite[Remark 1.12]{ANPE}).
\end{remark}

Lin and Hu \cite{Lin} investigated the following function
$$
{\mathcal{D}}_\sigma (t) :=\frac{D(\sigma+{\rm{i}}t)}{D(\sigma)}, \qquad 
D(s) := \prod_{p} \Bigl( 1- \frac{c(p)}{p^s} \Bigr)^{-1},
$$
where $c(p)$ are completely multiplicative non-negative coefficients. They proved that the function ${\mathcal{D}}_\sigma (t)$ is infinitely divisible when the product of $D(\sigma+{\rm{i}}t)$ converges absolutely. 

Afterwards, Aoyama and Nakamura \cite{ANPE} defined $m$-tuple compound Poisson zeta distributions on ${\mathbb{R}}$. Furthermore, they consider Multidimensional $\eta$-tuple $\varphi$-rank compound Poisson zeta distributions on ${\mathbb{R}}^d$. By applying the Kronecker's approximation theorem and Baker's theorem, they gave necessary and sufficient conditions for some polynomial Euler products to generate characteristic functions.

On the other hand, Aoyama and Nakamura \cite{AN12q} considered some two-variable finite Euler products and showed how they behave in view whether their corresponding normalized functions to be infinitely or quasi-infinitely divisible characteristic functions on $\R^2$. The quasi-infinitely divisibility is defined as follows.

\begin{definition}[Quasi-infinitely divisible distribution]
A distribution $\mu$ on $\rd$ is called {\it quasi-infinitely divisible} if it has a form of \eqref{INF} and the corresponding measure $\nu$ is a signed measure on $\rd$ with total variation measure $|\nu|$ which satisfy $\nu(\{0\}) =0$ and $\int_{\rd} (|x|^{2} \wedge 1) |\nu|(dx) < \infty$.
\end{definition}

Note that the triplet $(A,\nu ,\gamma)$ in this case is also unique if each component exists and that distributions on $\rd$ are quasi-infinitely divisible but not infinitely divisible if and only if the negative part of $\nu$ in the Jordan decomposition is not zero. The measure $\nu$ is called {\it quasi}-L\'evy measure and appeared in some books and papers, for example,  Gnedenko and Kolmogorov \cite[p.~81]{GK68}, Linnik and Ostrovskii \cite[Chap.~6, \S 7]{LO}, and others (see also Lindner and Sato \cite[Introduction]{LSm} or Sato \cite[Section 2.4]{Sato12}.)

\subsection{Aims of this paper}
In the present paper, we define zeta distributions on ${\mathbb{R}}^d$ generated by the following multidimensional polynomial Euler product
\begin{equation}\label{eq:aim14}
Z_E(\vs) = \prod_p \prod_{l=1}^\varphi \prod_{k=1}^\eta 
\biggl( 1 - \frac{\alpha_{lk}(p)}{p^{\langle \vc_l,\vs\rangle}} \biggr)^{-1} ,
\end{equation}
where $d,\varphi, \eta \in\N$, $\vc_l\in\mathbb{R}^d$, $\vs\in\mathbb{C}^d$, $\min_{1\le l\le \varphi} \Re \langle \vc_l,\vs\rangle >1$, $\alpha_{lk}(p) \in {\mathbb{C}}$ and $|\alpha_{lk}(p)|\le 1$ for $1 \le k \le \eta$ and $1 \le l \le \varphi$. The main aims of this paper are as follows.
\begin{enumerate}
\item Treat polynomial Euler products with complex coefficients.
\item Simplify the proofs of Theorems in \cite{ANPE}.
\item Consider applications to analytic number theory
\end{enumerate}

In association with the main aim (1), Aoyama and Nakamura \cite{ANPE} considered multidimensional polynomial Euler products with coefficients $\alpha_{lk}(p) \in \{ -1,0,1\}$. To adjust general number theory (see Section 1.2), we consider polynomial Euler products with complex coefficients $|\alpha_{lk}(p)|\le 1$ in the present paper. This change makes it possible to treat not only the case $\alpha_{lk}(p) \not \in \{ -1,0,1\}$ but also the case $a \vc_l = b \vc_k$, where $a$ and $b$ are some positive integers (see Section 4.4). 

The key of the proof of main theorem in \cite{ANPE} is Kronecker's approximation theorem. By using this theorem, the authors judged whether $|Z_E(\vs)/ Z_E(\Re(\vs))|\le 1$ or not. The method is interesting but not easy to understand. In this paper, we determine whether $Z_E(\vs)$ can generate a characteristic or not without Kronecker's approximation theorem (see Theorems \ref{th:d1} and \ref{th:dm1}). These theorems give simple proofs of many results in \cite{ANPE}.  

As applications to analytic number theory, we consider the value distribution of zeta functions in the region of absolute convergence. For example, we show that the Riemann zeta function $\zeta (s)$ satisfies inequalities (\ref{eq:ap2}) and (\ref{eq:ap3}) but the Dirichlet $L$-function $L(s)$ defined by (\ref{eq:defLep}) or (\ref{eq:defLds}) does not. Thus we can say that the value distribution of zeta and $L$-functions above are not same. It should be noted that $\zeta (s)$ can generates a characteristic function but $L(s)$ can not by Theorem \ref{th:cla1}. 

The paper is structured as follows. In Section 2, we define multidimensional polynomial Euler products with complex coefficients and give some important examples. Next we consider zeta distributions generated by one dimensional polynomial Euler products in Section 3. More precisely, we give necessary and sufficient conditions for polynomial Euler products with complex coefficients to generate infinitely divisible, quasi-infinitely divisible but non-infinitely divisible or not even characteristic functions when $\varphi=1$ in (\ref{eq:aim14}). Section 4 is the multidimensional case of Section 3. Namely, we consider the case $\varphi>1$ which is the main topic of this paper. We classify multidimensional polynomial Euler products into infinitely divisible, quasi-infinitely divisible but non-infinitely divisible, and not even characteristic functions by using Baker's theorem which is very famous in transcendental number theory. It should be noted that many examples of zeta distributions on ${\mathbb{R}}^d$ generated the polynomial Euler products are given in Sections 3 and 4. Finally, we consider applications to analytic number theory in Section 5. 

\section{Multidimensional polynomial Euler Products}

\subsection{Definition and properties}

Denote by $\Prime$ the set of all prime numbers.

\begin{definition}[Multidimensional polynomial Euler product, $Z_E(\vs)$]\label{def:EP}
Let $d,m\in\N$ and $\vs\in\mathbb{C}^d$.
For $\alpha_j(p) \in {\mathbb{C}}$, $|\alpha_j(p)|\le 1$ and non-zero vectors $\vc_j \in {\mathbb{R}}^d$, $1\le j \le m$, we define the following multidimensional polynomial Euler product given by
\begin{equation}
Z_E (\vs) = \prod_p \prod_{j=1}^m \biggl( 1 - \frac{\alpha_j(p)}{p^{\langle \vc_j,\vs\rangle}} \biggr)^{-1}.
\label{eq:def1}
\end{equation}
\end{definition}

Note that $\alpha_j(p)$ in \cite[Definition 2.1]{ANPE} is real number for any $1\le j\le m$ and $p\in\Prime$. In the present paper, we also consider the case $\alpha_l(p) \in {\mathbb{C}}$ in order to adjust general number theory. The polynomial Euler product with $d=1$ is commonly-used in number theory (see for example \cite{Steuding1}). This product converges absolutely when $\min_{1\le j\le m}\Re \langle \vc_j,\vs\rangle >1$ by the following lemma which coincides with \cite[Theorem 2.3]{ANPE} when $-1 \le \alpha_l(p) \le 1$. 

\begin{lemma}\label{lem:EPc} 
The product \eqref{eq:def1} converges absolutely and has no zeros in the region $\min_{1\le j\le m}\Re \langle \vc_j,\vs\rangle >1$. 
\end{lemma}

To prove this lemma, we quote the following proposition. 

\begin{proposition}[see {\cite[Theorem 15.4]{Rudin}}]
Suppose $\{ u_n \}$ is a sequence of bounded complex functions on a set $S$, such that $\sum |u_n(s)|$ converges uniformly on $S$. Then the product $f(s)=\prod_{n=1}^\infty (1+u_n(s))$ converges uniformly on $S$, and $f(s_0)=0$ at some $s_0 \in S$ if and only if $u_n(s_0)=-1$ for some $n \in {\mathbb{N}}$. 
\label{pro:co5.9}
\end{proposition}

\begin{proof}[Proof of Lemma \ref{lem:EPc}]
Put $v := \min_{1\le j\le m}\Re \langle \vc_j,\vs\rangle$.
Then, by the assumption $v>1$ and $|\alpha_j(p)| \le 1$ for any $p\in\Prime$ and $1\le j\le m$, we have
$$
\sum_{p} \bigl| \alpha_j(p) p^{-\langle \vc_j,\vs\rangle} \bigr| \le \sum_{p} p^{-v} \le \sum_{n\ge 2} n^{-v}
\le \int_1^\infty x^{-v} dx < \infty .
$$
Thus $\sum_{p} \alpha_j(p) p^{-\langle \vc_j,\vs\rangle}$ converges absolutely and uniformly on any compact subset of the region $\min_{1\le j\le m}\Re \langle \vc_j,\vs\rangle$$ >1$. By Proposition \ref{pro:co5.9}, the product \eqref{eq:def1} converges absolutely in the region $\min_{1\le j \le m}\Re \langle \vc_j,\vs\rangle >1$. We also have that $|1-\alpha_j(p) p^{-\langle \vc_j,\vs\rangle} |^{-1} >0$ for any  $p\in\Prime$ and $1\le j\le m$ when $\min_{1\le j\le m}\Re \langle \vc_j,\vs\rangle >1$, so that $\eqref{eq:def1}$ does not have zeros.
\end{proof}

Here and in the sequel, we define $\log Z_E (\vs)$ by the following Dirichlet series expansion
\begin{equation}\label{eq:9.19}
\log Z_E (\vs) := \sum_p \sum_{r=1}^\infty \sum_{j=1}^m \frac{1}{r} \alpha_l(p)^r p^{-r \langle \vc_j,\vs\rangle}
\end{equation}
in the region of absolute convergence $\min_{1\le j \le m}\Re \langle \vc_j,\vs\rangle >1$ (see e.g.~\cite[(9.19)]{Steuding1}). This formula will be used in some proofs of this paper. 

As mentioned in Section 1.2, the Riemann zeta function and Dirichlet $L$-functions have both the Euler products and the Dirichlet series expressions. Similarly, the polynomial Euler product with the condition all $\vc_l$ are the same also can be written by the Dirichlet series $\sum_{n=1}^\infty a(n)n^{-s}$. We quote some elementary properties for the coefficients $a(n)$ in the Dirichlet series expansion. 
\begin{proposition}[see {\cite[Lemma 2.2]{Steuding1}}]\label{lm:A}
Suppose that a function ${\mathcal{L}}(s)$ is given by 
$$
{\mathcal{L}}(s)= \sum_{n=1}^\infty \frac{a(n)}{n^s} = 
\prod_p \prod_{j=1}^m \Bigl( 1-\frac{\alpha_j(p)}{p^s} \Bigr)^{-1} , \qquad \sigma>1.
$$
Then $a(n)$ is multiplicative and
$$
a(n) = \prod_{p|n} \sum_{\substack{0 \le \theta_1, \ldots , \theta_m \\ \theta_1+ \cdots +\theta_m = \nu(n;p)}}
\prod_{j=1}^m \alpha_j(p)^{\theta_j},
$$
where $\nu(n;p)$ is the exponent of the prime $p$ in the prime factorization of the integer $n$. 
Moreover, if $|\alpha_j(p)| \le 1$ for $1 \le j \le m$ and all primes $p$, then $|a(n)| = O(n^{\varepsilon})$ for any $\varepsilon >0$, and vice versa.
\label{pro:st2.2}
\end{proposition}

\begin{remark}\label{rem:an}
It should be noted that one has $a(1)=1$ and $a(p)= \sum_{j=1}^m \alpha_j(p)$ from Proposition \ref{pro:st2.2}. This facts play an important role in Sections 3 and 4. 
\end{remark}

By using the proposition above, we obtain the following lemma.
\begin{lemma}\label{lem:mdst}
Let $\varphi,\eta \in {\mathbb{N}}$. Suppose $\min_{1\le l\le \varphi} \Re \langle \vc_l,\vs\rangle >1$. Then we have
$$
\prod_p \prod_{l=1}^\varphi \prod_{k=1}^\eta 
\biggl( 1 - \frac{\alpha_{lk}(p)}{p^{\langle \vc_l,\vs\rangle}} \biggr)^{-1} =
\prod_{l=1}^\varphi \sum_{n_l=1}^{\infty} \frac{a_l(n_l)}{n_l^{\langle \vc_l,\vs\rangle}} =
\sum_{n_1, \ldots ,n_\varphi=1}^\infty
\frac{a_1(n_1)}{n_1^{\langle \vc_1,\vs\rangle}} \cdots \frac{a_\varphi(n_\varphi)}{n_\varphi^{\langle \vc_\varphi,\vs\rangle}},
$$
where $a_l(n_l)$ is  multiplicative and written by
\begin{equation}
a_l(n_l) = \prod_{p|n_l} 
\sum_{\substack{0 \le \theta_1, \ldots , \theta_\eta \\ \theta_1+ \cdots +\theta_\eta = \nu(n_l;p)}}
\prod_{k=1}^\eta \alpha_{lk}(p)^{\theta_k}.
\label{eq:an}
\end{equation}
Furthermore, if $|\alpha_{lk}(p)| \le 1$ for $1 \le l \le \varphi$, $1 \le k \le \eta$ and all primes $p$, then $|a_l(n_l)| = O(n_l^{\varepsilon})$ for any $1 \le l \le \varphi$ and $\varepsilon >0$, and vice versa. In addition, the series $\prod_{l=1}^\varphi \sum_{n_l=1}^{\infty} a_l(n_l) n_l^{-\langle \vc_l,\vs\rangle}$ converges absolutely when $\min_{1\le l\le \varphi} \Re \langle \vc_l,\vs\rangle >1$. 
\end{lemma}
\begin{proof}
We only have to show the absolute convergence of the series $\prod_{l=1}^\varphi \sum_{n_l=1}^{\infty} a_l(n_l) n_l^{-\langle \vc_l,\vs\rangle}$ since the other statements are proved immediately from Proposition \ref{pro:st2.2}. By using Proposition \ref{pro:st2.2}, one has 
$$
\prod_p \prod_{k=1}^\eta \biggl( 1 - \frac{\alpha_{lk}(p)}{p^{\langle \vc_l,\vs\rangle}} \biggr)^{-1} =
\prod_p \prod_{k=1}^\eta 
\biggl( 1+ \sum_{j=1}^\infty \frac{\alpha_{lk}(p)^j}{p^{j\langle \vc_l,\vs\rangle}} \biggr) =
\sum_{n_l=1}^{\infty} \frac{a_l(n_l)}{n_l^{\langle \vc_l,\vs\rangle}},
$$
where $a_l(n_l)$ is defined by (\ref{eq:an}). The Dirichlet series above convergent absolutely when $\min_{1\le l\le \varphi} \Re \langle \vc_l,\vs\rangle >1$ since we have $|a_l(n_l)| = O(n_l^{\varepsilon})$ by Proposition \ref{pro:st2.2} and
$$
\sum_{n_l=1}^{\infty} \left| \frac{a_l(n_l)}{n_l^{\langle \vc_l,\vs\rangle}} \right| \le
\sum_{n_l=1}^{\infty} \frac{C_\varepsilon}{n_l^{\langle \vc_l,\vsig \rangle-\varepsilon}} \le C_\varepsilon +
C_\varepsilon \int_1^\infty x^{\varepsilon-\langle \vc_l,\vsig \rangle} dx
$$
for some $C_\varepsilon >0$. Therefore we obtain this lemma.
\end{proof}

\subsection{Examples of multidimensional polynomial Euler products}
Some simple examples of $Z_E(\vs)$ for $d=1$ are the following.

\begin{example}\label{ex:EP}
$(i)$ When $d=m=1$ and $\alpha(p)= p^{-\alpha}$, where $\alpha >0$, then
$$
Z_E (s_1) = \prod_p \frac{1}{1 -p^{-s_1-\alpha}} = \zeta (s_1+\alpha).
$$

\pn
$(ii)$ When $d=m=1$ and $\alpha(p)= -1$, then one has
$$
Z_E (s_1) = \prod_p \frac{1}{1 +p^{-s_1}} = \prod_p \frac{1 -p^{-s_1}}{1 -p^{-2s_1}}  = 
\frac{\zeta (2s_1)}{\zeta (s_1)}.
$$

\pn
$(iii)$ Let $\omega := e^{{\rm{i}}\pi/3}$. When $d=m=1$, $\vc=3$, $\alpha(p)=1$, or $d=1$, $\vc=1$, $m=3$ and $\alpha_1(p) =1$, $\alpha_2(p) =\omega$, $\alpha_3(p) =\omega^2$, then
$$
Z_E (s_1) = \prod_p \frac{1}{1 -p^{-3s_1}}  = \zeta (3s_1) = 
\prod_p \frac{1}{(1 -p^{-s_1})(1 - \omega p^{-s_1})(1 - \omega^2 p^{-s_1})}.
$$
\end{example}

Similarly, we have following examples for $d=2$ as a simple multidimensional case.

\begin{example}
$(iv)$ When $d=m=2$, $\vc_1=(1,0)$, $\vc_2=(1,2)$, $\alpha_1(p)=1$ and $\alpha_2(p)=\chi(p)$, then we have
$$
Z_E (\vs)=\prod_p \frac{1}{1 -p^{-s_1}} \frac{1}{1 -\chi(p)p^{-(s_1+2s_2)}}=\zeta(s_1)L(s_1+2s_2,\chi).
$$

\pn
$(v)$ When $d=2$, $m=3$, $\vc_1=(1,0)$, $\vc_2=(0,1)$, $\vc_3=(1,1)$, $\alpha_1(p)=1$, $\alpha_2(p)=\chi(p)$ and  $\alpha_3(p)= p^{-\alpha}$, where $\alpha >0$, then we have
$$
Z_E (\vs)=\prod_p \frac{1}{1 -p^{-s_1}} \frac{1}{1 -\chi(p)p^{-s_2}}\frac{1}{1 -p^{-s_1-s_2-\alpha}}=
\zeta(s_1)L(s_2,\chi)\zeta (s_1+s_2+\alpha).
$$
\end{example}

Let $d_k (n)$, $k=2,3,4,\ldots$, denote the number of ways of expressing $n$ as a product of $k$ factors, expression with the same factors in a different order being counted as different. 
\begin{example}
$(vi)$ It is known that (see, for example \cite[(1.2.2)]{Tit})
$$
\prod_p \bigl( 1 - p^{-s} \bigr)^{-k} = \zeta^k (s) = 
\sum_{m_1=1}^\infty \frac{1}{m_1^s} \cdots \sum_{m_k=1}^\infty \frac{1}{m_k^s} =
\sum_{n=1}^\infty \frac{1}{n^s} \sum_{m_1 \cdots m_k=n} 1 = \sum_{n=1}^\infty \frac{d_k(n)}{n^s}.
$$
Moreover, we have $d_k(n) = O(n^{\varepsilon})$ by Proposition \ref{pro:st2.2}. 
\end{example}

Now we define the following Dirichlet $L$-function $L(s)$ by
\begin{equation}
L(s) := \prod_{p \,:\, {\rm{odd}}} \Bigl( 1-(-1)^{\frac{p-1}{2}}p^{-s} \Bigr)^{-1}, \qquad \sigma >1.
\label{eq:defLep}
\end{equation}
It is well-known that $L(s)$ is also expressed by 
\begin{equation}
L(s) = \sum_{n=1}^\infty \frac{\chi_{-4}(n)}{n^s}, \qquad \chi_{-4}(n) :=
\begin{cases}
1 & n \equiv 1 \mod 4,\\
-1 &  n \equiv 3 \mod 4,\\
0  & n \equiv 0,2 \mod 4.
\end{cases}
\label{eq:defLds}
\end{equation}
Furthermore, let $\mathbb{Q}({\rm i})$ be a quadratic field of discriminant $-4$.
The Dedekind zeta function of $\mathbb{Q}({\rm i})$ is a function of a complex variables $s=\sigma +{\rm i}t$, for $\sigma >1$ given by
$$
\zeta_{{\mathbb{Q}}({\rm i})} (s) := \zeta (s) L(s).
$$
\begin{example}\label{exa:dqi1}
$(vii)$ It is known that (see, for example \cite[p.~221]{Cohen})
$$
\zeta_{{\mathbb{Q}}({\rm i})} (s) = \frac{1}{4} \sum_{(m,n) \in {\mathbb{Z}}^2 \setminus (0,0)} \frac{1}{(m^2+n^2)^s} = 
\sum_{n=1}^\infty \frac{a^\#(n)}{n^s},
$$
where $a^\#(n)$ is nonnegative definite coefficient written as
\begin{equation}\label{eq:ash}
a^\#(n) := \frac{1}{4} \# \{ (m_1,m_2) \in {\mathbb{Z}}^2 : m_1^2+m_2^2 =n\} = \sum_{m \mid n} \chi_{-4}(m),
\end{equation}
where the sum $\sum_{m \mid n}$ takes all positive divisors of $n$. Moreover, it holds that $a^\#(n) = O(n^{\varepsilon})$ by Proposition \ref{pro:st2.2}. 
\end{example}

\section{Zeta distributions generated by polynomial Euler products}
Here and in the sequel, we put 
$$
\vs:=\vsig +{\rm i}\vt, \qquad \Vec{\sigma}, \vt\in\rd.
$$
In this section, we only consider the case when $\Re \langle \vc,\vs\rangle >1$ where $\vc := \vc_1 = \cdots = \vc_m \in {\mathbb{R}}^d$ in (\ref{eq:def1}), namely, we only treat the following type of polynomial Euler products
\begin{equation}
Z_E(\vs) = \prod_p \prod_{k=1}^\eta \biggl( 1 - \frac{\alpha_{k}(p)}{p^{\langle \vc,\vs\rangle}} \biggr)^{-1} ,
\label{eq:defpe3}
\end{equation}
where $\alpha_k(p) \in {\mathbb{C}}$, $|\alpha_k(p)|\le 1$, $1 \le k \le \eta$. When $-1 \le \alpha_k(p) \le 1$, this function coincides with the polynomial Euler product treated in \cite[Section 3]{ANPE}.

In the view of Proposition \ref{pro:st2.2}, $Z_E(\vs)$ is also written by as follows.
\begin{equation}
Z_E(\vs) = \sum_{n=1}^{\infty} \frac{a(n)}{n^{\langle \vc,\vs\rangle}} , \qquad
a(n) = \prod_{p|n} 
\sum_{\substack{0 \le \theta_1, \ldots , \theta_\eta \\ \theta_1+ \cdots +\theta_\eta = \nu(n;p)}}
\prod_{k=1}^\eta \alpha_k(p)^{\theta_k}.
\label{eq:defsr3}
\end{equation}
We have to note that the series above converges absolutely when $\Re \langle \vc,\vs\rangle >1$ by the fact that $a(n) = O(n^{\varepsilon})$ proved in Lemma \ref{lm:A}.

For $\vsig$ satisfying $\langle \vc,\vsig \rangle >1$, we define a normalized function
\begin{equation}
f_{\vsig}\left(\vt\,\right):=\frac{Z_E\left(\vsig +{\rm i}\vt\,\right)}{Z_E(\vsig)}.
\label{eq:nor1}
\end{equation}
Thus the zeta distribution defined by the characteristic function above is essentially one dimensional.

\subsection{Infinitely divisible or not}
We have the following.
\begin{theorem}\label{th:id1}
Let $\alpha_k(p) \in {\mathbb{C}}$, $|\alpha_k(p)|\le 1$ for any $p\in\Prime$ and $1 \le k \le \eta$ in (\ref{eq:defpe3}). Then $f_{\vsig}$ is an infinitely divisible characteristic function if and only if $\sum_{k=1}^\eta \alpha_k(p)^r \ge 0$ for all $r \in {\mathbb{N}}$ and $p\in\Prime$.
Moreover, when $\sum_{k=1}^\eta \alpha_k(p)^r \ge 0$ for all $r \in {\mathbb{N}}$ and $p\in\Prime$, $f_{\vsig}$ is a compound Poisson characteristic function with its finite L\'evy measure $N_{\vsig}$ on $\rd$ given by 
\begin{equation}\label{eq:lm1}
N_{\vsig} (dx) = \sum_p \sum_{r=1}^{\infty} \sum_{k=1}^\eta \frac{1}{r}
\alpha_k(p)^r p^{-r\langle\vc,\vsig\rangle} \delta_{\log p^r \vc} (dx).
\end{equation}
\end{theorem}

To prove the theorem above, we show the following lemma. 
\begin{lemma}\label{lem:lm1}
Let $\alpha_k(p) \in {\mathbb{C}}$, $|\alpha_k(p)|\le 1$ for any $p\in\Prime$ and $1 \le k \le \eta$. Then $N_{\vsig}$ is a complex measure on $\rd$ with total variation measure $|N_{\vsig}|$ satisfying $N_{\vsig}(\{0\}) =0$ and $\int_{\rd} (|x| \wedge 1) |N_{\vsig}|(dx) < \infty$.
\end{lemma}
\begin{proof}
Recall that $\log Z_E (\vs)$ defined by (\ref{eq:9.19}). From Lemma \ref{lem:EPc}, for $\vt\in\rd$, the normalized function $f_{\vsig}(\vt)$ introduced in (\ref{eq:nor1}) converges when $\langle \vc,\vsig\rangle >1$. Then we have
\begin{equation*}
\begin{split}
&\log f_{\vsig}(\vt)= \log \frac{Z_E (\vsig + {\rm i}\vt)}{Z_E (\vsig)} 
= \sum_p \sum_{k=1}^\eta \sum_{r=1}^{\infty} 
\frac{1}{r}\alpha_k(p)^r p^{-r\langle\vc,\vsig\rangle} \bigl(p^{-r\langle\vc,{\rm i}\vt\rangle} -1\bigr)  \\ 
= & \sum_p \sum_{r=1}^{\infty} \sum_{k=1}^\eta \frac{1}{r} 
\alpha_k(p)^r p^{-r\langle\vc,\vsig\rangle} \bigl(e^{-r\langle\vc,{\rm i}\vt\rangle \log p} -1\bigr) 
=\int_{\rd} (e^{-\langle{\rm i}\vt,x\rangle}-1)N_{\vsig}(dx),
\end{split}
\end{equation*}
where $N_{\vsig}$ is expressed as (\ref{eq:lm1}) since we have $e^{-r\langle\vc,{\rm i}\vt\rangle \log p}=\int_{\rd} e^{-\langle{\rm i}\vt,x\rangle} \delta_{\log p^r \vc} (dx)$. Now put $v := \langle \vc,\vsig\rangle>1$. By the assumption $\alpha_k(p) \in {\mathbb{C}}$, $|\alpha_k(p)|\le 1$ for any $p\in\Prime$ and $1 \le k \le \eta$, it holds that
\begin{equation*}
\begin{split}
N_{\vsig} (\rd) \le & \int_{\rd}\sum_p \sum_{r=1}^{\infty} \sum_{k=1}^\eta \frac{1}{r}
|\alpha_k(p)|^r p^{-r\langle\vc,\vsig\rangle} \delta_{\log p^r \vc} (dx)
= \eta \sum_p \sum_{r=1}^{\infty} \frac{1}{r} p^{-r\langle\vc,\vsig\rangle} \\ 
\le & \, \eta \sum_p \sum_{r=1}^{\infty} p^{-r\langle\vc,\vsig\rangle} \le \eta
\sum_{n=2}^{\infty} \sum_{r=1}^{\infty} n^{-rv} = \eta \sum_{n=2}^{\infty} \frac{n^{-v}}{1-n^{-v}} 
 \le 2 \eta \sum_{n=2}^{\infty} n^{-v} \\
\le & \, 2\eta \zeta (v) <\infty.
\end{split}
\end{equation*}
It is also easy to see that the measure $N_{\vsig}$ satisfies $\int_{|x|<1}|x|N_{\vsig}(dx)\le N_{\vsig}(\rd)<\infty$. 
\end{proof}

\begin{proof}[Proof of Thereom \ref{th:id1}]
First suppose $\sum_{k=1}^\eta \alpha_k(p)^r \ge 0$ for all $r \in {\mathbb{N}}$ and $p\in\Prime$. In this case, we can see that $N_{\vsig}$ is a measure on $\rd$ with $N_{\vsig}(\{0\}) =0$ and $\int_{\rd} (|x| \wedge 1) N_{\vsig}(dx) < \infty$ by Lemma \ref{lem:lm1}. Hence $f_{\vsig}$ is an infinitely divisible characteristic function.

Next we suppose that there exists a pair of $r_0 \in {\mathbb{N}}$ and $p_0 \in\Prime$ such that $\sum_{k=1}^\eta \alpha_k(p_0)^{r_0} \in K$, where $K := \{ z \in {\mathbb{C}} : |z| \le \eta, z \not \in [0,\eta] \}$. Let $r_1,r_2 \in {\mathbb{N}}$ and $p_1, p_2 \in\Prime$. By the fundamental theorem of arithmetic, we have
\begin{equation}
r_1 \log p_1 = r_2 \log p_2 \quad \mbox{ if and only if } \quad r_1=r_2 \mbox{ and } p_1=p_2 .
\label{eq:fta}
\end{equation}
Therefore, one has $\delta_{\log {p_1}^{r_1} \vc} (dx) = \delta_{\log {p_2}^{r_2} \vc} (dx)$ if and only if $r_1=r_2$ and $p_1=p_2$. Hence the normalized function $f_{\vsig}$ is not an infinitely divisible characteristic function by the (not measure but) complex signed measure $\sum_{k=1}^\eta \alpha_k(p_0)^{r_0} \delta_{\log {p_0}^{r_0} \vc}$. 
\end{proof}

\begin{remark}
If $p_1$ or $p_2$ is not a prime number, the statement (\ref{eq:fta}) is not true. For example, when $p_1=2$ and $p_2=8$, we have $6\log p_1 = 2 \log p_2 = \log 64$. Therefore, the Euler product, namely, the product of prime numbers, plays an very important role in the proof of Theorem \ref{th:id1}. 
\end{remark}

\begin{example}\label{ex:id1}
Let $p_n$ be the $n$-th prime number.
\pn$(i)$
The functions $\prod_n (1-{\rm{i}}^np_n^{-s})^{-1}$ and $\prod_n (1-(-{\rm{i}})^np_n^{-s})^{-1}$ are not to generate infinitely divisible characteristic functions. 
\pn$(ii)$
Let $\eta =4$, $\alpha_1(p_n) = \alpha_2 (p_n) =1$, $\alpha_3 (p_n) = {\rm{i}}^n$ and $\alpha_4 (p_n) = (-{\rm{i}})^n$. Then we have $\sum_{k=1}^4 \alpha_k(p_n)^r \ge 0$ for any $n,r \in {\mathbb{N}}$. Hence the Euler product $\prod_n (1-p_n^{-s})^{-2}(1-{\rm{i}}^np_n^{-s})^{-1}(1-(-{\rm{i}})^np_n^{-s})^{-1}$ is to generate an infinitely divisible characteristic function.
\pn$(iii)$ Let $\eta=3$, $\alpha_1(p) \ge 0$, $\alpha_2 (p)= \overline{\alpha_3 (p)}$ and $2|\alpha_2(p)| \le \alpha_1(p)$ for any $p \in\Prime$. Then we have $\sum_{k=1}^3 \alpha_k(p)^r \ge 0$ for any $p \in\Prime$ and $r \in \N$. 
\end{example}

\subsection{Distribution or not}
We have the following.
\begin{theorem}\label{th:d1}
Let $\alpha_k(p) \in {\mathbb{C}}$, $|\alpha_k(p)|\le 1$ for any $p\in\Prime$ and $1 \le k \le \eta$ in (\ref{eq:defpe3}). Then $f_{\vsig}$ is a characteristic function if and only if $a(n) \ge 0$ for all $n \in {\mathbb{N}}$, where the sequence $a(n)$ is defined by (\ref{eq:defsr3}).
\end{theorem}
In order to prove this theorem, we define a generalized Dirichlet $L$ random variable $X_{\vsig}$ with probability distribution on ${\mathbb{R}}^d$ given by
\begin{equation}
{\rm Pr} \bigl(X_{\vsig}= - \log n \vc \bigr)= \frac{1}{Z_E(\vsig)} \frac{a(n)}{n^{\langle \vc, \vsig \rangle}},
\qquad a(n) \ge 0 , \quad n \in \N.
\label{eq:gdsd}
\end{equation}
It is easy to see that these distributions are probability distributions (see also Lemma \ref{lem:1}) since $a(n)n^{-\langle \vc, \vsig \rangle} \ge 0$ for each $n \in {\mathbb{N}}$, and 
$$
\sum_{n=1}^{\infty} \frac{a(n)n^{-\langle \vc, \vsig \rangle}}{Z_E(\vsig)} = \frac{1}{Z_E(\vsig)} 
\sum_{n=1}^\infty \frac{a(n)}{n^{\langle \vc, \vsig \rangle}}=\frac{Z_E(\vsig)}{Z_E(\vsig)}=1 .
$$
Note that these distributions belong to a special case of multidimensional Shintani zeta distribution defined by Aoyama and Nakamura \cite{AN12s}.

We immediately obtain Theorem \ref{th:d1} by using the following Lemmas \ref{lem:1} and \ref{lem:2}. It should be mentioned that these lemmas already have been in \cite[Lemmas 2.2 and 2.3]{Nakamura12} when $d=1$.

\begin{lemma}\label{lem:1}
Let $X_{\vsig}$ be a generalized Dirichlet $L$ random variable. Then its characteristic function $f_{\vsig}$ is given by (\ref{eq:nor1}).
\end{lemma}
\begin{proof}
By the definition, we have, for any $\vt \in {\mathbb{R}}^d$, 
\begin{equation*}
\begin{split}
f_{\vsig}(\vt) = & \sum_{n=1}^\infty 
e^{{\rm i} \langle \vt, - \log n \vc \rangle} \frac{a(n)n^{-\langle \vc, \vsig \rangle}}{Z_E(\vsig)} =
\frac{1}{Z_E(\vsig)} \sum_{n=1}^\infty 
\frac{e^{-{\rm i} \langle \vc, \vt \rangle \log n}a(n)}{n^{\langle \vc, \vsig \rangle}} \\ = &
\frac{1}{Z_E(\vsig)} \sum_{n=1}^\infty  
\frac{a(n)}{n^{\langle \vc, \vsig \rangle+ {\rm i} \langle \vc, \vt \rangle}} =
\frac{Z_E(\vsig +{\rm i}\vt)}{Z_E(\vsig)}.
\end{split}
\end{equation*}
This equality implies the lemma.
\end{proof}

\begin{lemma}\label{lem:2}
Suppose that there exists $m \in {\mathbb{N}}$ such that $a(m) \in {\mathbb{C}} \setminus {\mathbb{R}}_{\ge 0}$. Then the function $Z_E(\vsig +{\rm i}\vt)/Z_E(\vsig)$ is not a characteristic function. 
\end{lemma}
\begin{proof}
Let ${\mathbb{N}}_+$ be the set of integers $n$ such that $a(n)\ge 0$ and ${\mathbb{N}}_+^c$ be the set of integers $m$ such that $a(m) \in {\mathbb{C}} \setminus {\mathbb{R}}_{\ge 0}$. From the view of (\ref{eq:gdsd}), we have 
\begin{equation*}
\begin{split}
& \frac{Z_E(\vsig +{\rm i}\vt)}{Z_E(\vsig)} = \frac{1}{Z_E(\vsig)} 
\sum_{n=1}^\infty \frac{a(n)}{n^{\langle \vc, \vsig \rangle+{\rm i}\langle \vc, \vt \rangle}}= 
\frac{1}{Z_E(\vsig)} \sum_{m \in {\mathbb{N}}_+^c} 
\frac{a(m)}{m^{\langle \vc, \vsig \rangle+{\rm i}\langle \vc, \vt \rangle}} + 
\frac{1}{Z_E(\vsig)} \sum_{n \in {\mathbb{N}}_+} 
\frac{a(n)}{n^{\langle \vc, \vsig \rangle+{\rm i}\langle \vc, \vt \rangle}} \\ &=
\frac{1}{Z_E(\vsig)} \int_{{\mathbb{R}}^d} e^{{\rm i}\langle \vt, x \rangle} 
\sum_{m \in {\mathbb{N}}_+^c} \frac{a(m)}{m^{\langle \vc, \vsig \rangle}} \delta_{- \log m \vc} (dx) + 
\frac{1}{Z_E(\vsig)} \int_{{\mathbb{R}}^d} e^{{\rm i} \langle \vt, x \rangle} 
\sum_{n \in {\mathbb{N}}_+} \frac{a(n)}{n^{\langle \vc, \vsig \rangle}} \delta_{- \log n \vc} (dx) \\ &=
\frac{1}{Z_E(\vsig)} \sum_{m \in {\mathbb{N}}_+^c} \frac{a(m)}{m^{\langle \vc, \vsig \rangle}} 
\int_{{\mathbb{R}}^d} e^{{\rm i}\langle \vt, x \rangle} \delta_{- \log m \vc} (dx) + 
\frac{1}{Z_E(\vsig)} \sum_{n \in {\mathbb{N}}_+} \frac{a(n)}{n^{\langle \vc, \vsig \rangle}} 
\int_{{\mathbb{R}}^d} e^{{\rm i}\langle \vt, x \rangle} \delta_{- \log n \vc} (dx).
\end{split}
\end{equation*}
Note that $a(1)=1$ by Remark \ref{rem:an}. By the set of integers $m$ such that $a(m)< 0$, 
\begin{equation}
\frac{1}{Z_E(\vsig)} \sum_{m \in {\mathbb{N}}_+^c} \frac{a(m)}{m^{\langle \vc, \vsig \rangle}} 
\delta_{- \log m \vc} (dx) + 
\frac{1}{Z_E(\vsig)} \sum_{n \in {\mathbb{N}}_+} \frac{a(n)}{n^{\langle \vc, \vsig \rangle}} 
\delta_{- \log n \vc} (dx)
\label{me:1}
\end{equation}
is not a measure but a complex signed measure. Moreover, we have
$$
\frac{1}{Z_E(\vsig)} \int_{{\mathbb{R}}^d} \sum_{n=1}^\infty \left| \frac{a(n)}{n^{\langle \vc, \vsig \rangle}} 
\right| \delta_{- \log n \vc} (dx) = 
\frac{1}{Z_E(\vsig)} \sum_{n=1}^\infty \frac{|a(n)|}{n^{\langle \vc, \vsig \rangle}} < \infty
$$
by the assumption $\Re \langle \vc,\vs\rangle >1$ and the fact that $a(n) = O(n^{\varepsilon})$ (see the proof of Lemma \ref{lem:mdst}). Hence the complex signed measure (\ref{me:1}) has finite total variation. It is known that any complex signed measure with finite total variation is uniquely determined by the Fourier transform. Therefore, $Z_E(\vsig +{\rm i}\vt)/Z_E(\vsig)$ is not a characteristic function.
\end{proof}

\begin{remark}\label{rem:ans1}
Let $n$ be a integer written by $n=p_1^{r_1} \cdots p_j^{r_j}$, where $p_1, \ldots , p_j$ are distinct prime numbers and $r_1, \ldots , r_j \in {\mathbb{N}}$. By Lemma \ref{lm:A}, any coefficient $a(n)$ in the Dirichlet series (\ref{eq:defsr3}) is multiplicative. Namely, it holds that $a(n) = a(p_1^{r_1}) \cdots a(p_j^{r_j})$. Therefore, we have $a(n) \ge 0$ for any $n \in {\mathbb{N}}$ if and only if 
$$
a(p^r)=\sum_{\substack{0 \le \theta_1, \ldots , \theta_\eta \\ \theta_1+ \cdots +\theta_\eta = r}}\prod_{k=1}^\eta \alpha_k(p)^{\theta_k} 
$$
is non-negative for any $r \in {\mathbb{N}}$ and $p\in\Prime$. 
\end{remark}

\begin{example}
We use the same notation appeared in Example \ref{exa:dqi1}. 
\pn$(v)$
Let $\alpha (2) =0$ and $\alpha(p) = (-1)^{(p-1)/2}$ for any $p>3$. Then the function $L(s)$ defined by (\ref{eq:defLep}) or (\ref{eq:defLds}) is not to generate a characteristic function.
\pn$(vi)$
Let $\eta =2$, $\alpha_1 (2) = 1$, $\alpha_2 (2) =0$, $\alpha_1(p) =1$ and $\alpha_2(p) = (-1)^{(p-1)/2}$ for any $p>3$. Then we have $a(n)=\sum_{d \mid n} \chi_{-4}(d) \ge 0$ from Example \ref{exa:dqi1}. Hence $\zeta_{{\mathbb{Q}}({\rm i})} (s)= \zeta (s) L(s)$ is to generate a characteristic function.
\pn$(vii)$ Let $\prod_{k=1}^\eta (1-\alpha_k(p)p^{-s}) = (1-\beta_1(p)p^{-s}- \cdots -\beta_\eta (p) p^{-\eta s})$, where $\beta_k(p) \ge 0$, $1 \le k \le \eta$ and $\sup_{\Re (s) \ge 1}|\sum_{k=1}^\eta \beta_k(p)p^{-ks}|<1$. Then we have $a(n) \ge 0$ for any $n \in \N$. This is proved the following equation
$$
\prod_{k=1}^\varphi (1-\alpha_k(p)p^{-s})^{-1} = 
\biggl(1-\sum_{k=1}^\eta \frac{\beta_k(p)}{p^{ks}} \biggr)^{-1} =
1+\sum_{j=1}^\infty \biggl(\sum_{k=1}^\eta \frac{\beta_k(p)}{p^{ks}} \biggr)^j.
$$
For instance, when $\eta=3$, $\alpha_1(p) = 1$ and $5\alpha_2(p) = -1-3{\rm{i}}$, $5\alpha_3(p) =-1+3{\rm{i}}$, we have $\prod_{k=1}^3 (1-\alpha_k(p)p^{-s})^{-1} = (1- (3/5) p^{-s} - (2/5) p^{-3s})^{-1}$. 
\end{example}

\subsection{Classification}
Let $F_E$ be the set of normalized functions of $Z_E(\vsig +{\rm i}\vt)/Z_E(\vsig)$, where $Z_E(\vsig +{\rm i}\vt)$ is defined by (\ref{eq:defpe3}) or (\ref{eq:defsr3}). Moreover, let $\widehat{I\!D}$, $\widehat{I\!D^0}$ and $\widehat{N\!D}$ be the class of infinitely divisible characteristic functions, the class of quasi-infinitely divisible but non-infinitely divisible characteristic functions, and the class of functions not even characteristic functions, respectively. Then we have the following theorem.
\begin{theorem}\label{th:cla1}
Let $f_{\vsig} \in F_E$, $\alpha_k(p) \in {\mathbb{C}}$, $|\alpha_k(p)|\le 1$ for any $p\in\Prime$ and $1 \le k \le \eta$ in (\ref{eq:defpe3}). Then we have\\
$(I)$ $f_{\vsig} \in \widehat{I\!D}$ if and only if $\sum_{k=1}^\eta \alpha_k(p)^r \ge 0$ for all $r \in {\mathbb{N}}$ and $p\in\Prime$.\\
$(II)$ $f_{\vsig} \in \widehat{I\!D^0}$ if and only if $a(n) \ge 0$ for all $n \in {\mathbb{N}}$, and there exist $r_0 \in {\mathbb{N}}$ and $p_0 \in\Prime$ such that $\sum_{k=1}^\eta \alpha_k(p_0)^{r_0} <0$.\\
$(III)$ $f_{\vsig} \in \widehat{N\!D}$ if and only if there exists $m \in {\mathbb{N}}$ such that $a(m) \in {\mathbb{C}} \setminus {\mathbb{R}}_{\ge 0}$. \\
$(IV)$ $F_E= \widehat{I\!D} \biguplus \widehat{I\!D^0} \biguplus \widehat{N\!D}$. 
\end{theorem}

In order to prove the theorem above, we only have to show the following lemma since we obtain the statement $(IV)$ is proved by $(I)$, $(II)$ and $(III)$, and the statements $(I)$ and $(III)$ are proved by Theorems \ref{th:id1} and \ref{th:d1}, respectively.
\begin{lemma}\label{lem:cla1}
Let $\alpha_k(p) \in {\mathbb{C}}$, $|\alpha_k(p)|\le 1$ for any $p\in\Prime$ and $1 \le k \le \eta$ in (\ref{eq:defpe3}). Suppose that $a(n) \ge 0$ for all $n \in {\mathbb{N}}$, where $a(n)$ is defined by (\ref{eq:defsr3}). Then $\sum_{k=1}^\eta \alpha_k(p)^r$ is real for any $r \in {\mathbb{N}}$ and $p\in\Prime$. 
\end{lemma}
\begin{proof}
By (\ref{eq:defsr3}) and the assumption that $a(n) \ge 0$ for all $n \in {\mathbb{N}}$, namely, $a(n) \in {\mathbb{R}}$ for any $n \in {\mathbb{N}}$, we have $\overline{Z_E(\vsig+\vt)} = Z_E (\vsig-\vt)$, where $\overline{z}$ is the complex conjugate of $z \in {\mathbb{C}}$. In the view point of Lemma \ref{lem:lm1}, we have
\begin{equation*}
\begin{split}
\log \overline{Z_E(\vsig+\vt)} =& \sum_p \sum_{r=1}^{\infty} \sum_{k=1}^\eta \frac{1}{r} \overline{\alpha_k(p)^r p^{-r\langle\vc,\vsig\rangle-r\langle\vc,{\rm i}\vt\rangle}} =
\sum_p \sum_{r=1}^{\infty} \sum_{k=1}^\eta \frac{1}{r} \overline{\alpha_k(p)} {}^r 
p^{-r\langle\vc,\vsig\rangle + r\langle\vc,{\rm i}\vt\rangle}, \\
\log Z_E (\vsig-\vt) =& \sum_p \sum_{r=1}^{\infty} 
\sum_{k=1}^\eta \frac{1}{r} \alpha_k(p)^r p^{-r\langle\vc,\vsig\rangle + r\langle\vc,{\rm i}\vt\rangle} .
\end{split}
\end{equation*}
Therefore, we obtain $\sum_{k=1}^\eta \overline{\alpha_k(p)} {}^r = \sum_{k=1}^\eta \alpha_k(p)^r$ for any $r \in {\mathbb{N}}$ and $p\in\Prime$ by uniqueness theorem for Dirichlet series (see for example \cite[Theorem 11.3]{Apo}). 
\end{proof}

\begin{example}\label{exa:q}
Let $\eta=3$, $d= \vc =1$, $\alpha_1(2)=1$, $\alpha_2(2)=-\alpha_3(2)={\rm{i}}$, and $\alpha_1(p) = \alpha_2(p) = \alpha_3(p) = 0$ for any $p>3$. Then one has 
$$
\sum_{k=1}^3 \alpha_k (2)^{4j-3}=1, \quad \sum_{k=1}^3 \alpha_k (2)^{4j-2}=-1, \quad
\sum_{k=1}^3 \alpha_k (2)^{4j-1}=1, \quad \sum_{k=1}^3 \alpha_k (2)^{4j}=3
$$
for any $j \in {\mathbb{N}}$. On the other hand, we have
\begin{equation*}
\begin{split}
Z_Q (s) &:= \frac{1}{(1-2^{-s})(1-{\rm{i}}2^{-s})(1+{\rm{i}}2^{-s})} = \frac{1}{(1-2^{-s})(1+2^{-2s})} = 
\frac{1+2^{-s}}{(1-2^{-2s})(1+2^{-2s})} \\ &= \frac{1+2^{-s}}{1-2^{-4s}}
= \bigl(1+2^{-s}\bigl) \sum_{n=0}^\infty 2^{-4ns} = 
\sum_{n=0}^\infty \bigl( 2^{-4ns} + 2^{-(4n+1)s} \bigr) \\
&= \sum_{n=1}^\infty \frac{a(n)}{n^s}, \qquad  a(n) :=
\begin{cases}
1 & n= 2^{4j}, 2^{4j+1}, \,\,\, j=0,1,2,3,\ldots , \\
0 & \mbox{otherwise}.
\end{cases}
\end{split}
\end{equation*}
Hence this function satisfies $\sum_{k=1}^3 \alpha_k (2)^{4j-2} <0$ for any $j \in {\mathbb{N}}$ and $a(n) \ge 0$ for all $n \in {\mathbb{N}}$. Therefore, there exists a $f_{\sigma} \in F_E$ such that $f_{\sigma} \in \widehat{I\!D^0}$.
\end{example}
\begin{remark}\label{re:col}
Let $n \in {\mathbb{N}} \cup \{0\}$ and put
$$
F_n (\sigma, t):= \frac{(1-2^{-\sigma})^n(1-{\rm{i}}2^{-\sigma})(1+{\rm{i}}2^{-\sigma})}{(1-2^{-\sigma-{\rm{i}}t})^n(1-{\rm{i}}2^{-\sigma-{\rm{i}}t})(1+{\rm{i}}2^{-\sigma-{\rm{i}}t})} .
$$
Then it holds that
$$
\sum_{k=1}^{n+2} \alpha_k (2)^{4j-3}=n, \quad \sum_{k=1}^{n+2} \alpha_k (2)^{4j-2}= n-2, \quad
\sum_{k=1}^{n+2} \alpha_k (2)^{4j-1}=n, \quad \sum_{k=1}^{n+2} \alpha_k (2)^{4j}= n+2
$$
for any $j \in {\mathbb{N}}$. Hence we have $F_1 \in \widehat{I\!D^0}$ (see Example \ref{exa:q}) and $F_n \in \widehat{I\!D}$ when $n \ge 2$. Moreover, one has $F_0 \in \widehat{N \! D}$. This is proved by
$$
F_0 (\sigma, t)= \frac{(1-{\rm{i}}2^{-\sigma})(1+{\rm{i}}2^{-\sigma})}{(1-{\rm{i}}2^{-\sigma-{\rm{i}}t})(1+{\rm{i}}2^{-\sigma-{\rm{i}}t})} = \frac{1+2^{-2\sigma}}{1+2^{-2\sigma-2{\rm{i}}t}} =
\frac{\sum_{j=0}^\infty (-4)^{-j(\sigma+{\rm{i}}t)}}{\sum_{j=0}^\infty (-4)^{-j\sigma}} 
$$
and (III) of Theorem \ref{th:cla1}. 
\end{remark}

Hence one has $\widehat{I\!D^0} \ne \emptyset$ in general (see also Example \ref{exa:123}). However, the following proposition implies $\widehat{I\!D^0} = \emptyset$ when $\alpha_k(p) \in \{ 0,1,-1\}$. This result is proved in \cite[Theorem 3.1]{ANPE}. We give a simpler proof here by using Theorems \ref{th:cla1}. 
\begin{proposition}[{\cite[Theorem 3.1]{ANPE}}]
Let $\alpha_k(p) \in \{ 0,1,-1\}$ for any $p\in\Prime$ and $1 \le k \le \eta$ in (\ref{eq:defpe3}). Then it holds that $F_E= \widehat{I\!D} \biguplus \widehat{N\!D}$. Namely, we have $f_{\vsig} \in \widehat{N\!D}$ if and only if there exist $r_0 \in {\mathbb{N}}$ and $p_0 \in\Prime$ such that $\sum_{k=1}^\eta \alpha_k(p_0)^{r_0} <0$. 
\end{proposition}
\begin{proof}
We only have to show $\widehat{I\!D^0} = \emptyset$ in this case. Suppose that $f_{\vsig} \not \in \widehat{I\!D}$, namely there exist $r_0 \in {\mathbb{N}}$ and $p_0 \in\Prime$ such that $\sum_{k=1}^\eta \alpha_k(p_0)^{r_0} <0$. By the assumption $\alpha_k(p) \in \{ 0,1,-1\}$, we have $\sum_{k=1}^\eta \alpha_k(p_0)^{2j} \ge 0$ and $\sum_{k=1}^\eta \alpha_k(p_0) = \sum_{k=1}^\eta \alpha_k(p_0)^{2j-1}$ for any $j \in {\mathbb{N}}$ (see also \cite[Lemma 3.2]{ANPE}). Hence we can assume that $\sum_{k=1}^\eta \alpha_k(p_0) <0$. On the other hand, we have $a(p) = \sum_{k=1}^\eta \alpha_k(p)$ from Remark \ref{rem:an}. Therefore, by using Theorem \ref{th:d1}, we obtain $f_{\vsig} \in \widehat{N\!D}$ since we have $a(p_0)<0$ if $\sum_{k=1}^\eta \alpha_k(p_0) <0$. Hence one has $f_{\vsig} \in \widehat{N\!D}$ when $f_{\vsig} \not \in \widehat{I\!D}$, equivalently $f_{\vsig} \in \widehat{I\!D^0} \biguplus \widehat{N\!D}$, under the assumption $\alpha_k(p) \in \{ 0,1,-1\}$. 
\end{proof}

\begin{example}\label{ex:eta1}
When $\eta=1$, one has $\widehat{I\!D^0} = \emptyset$. This is proved as follows. Suppose there exist $r_0 \in {\mathbb{N}}$ and $p_0 \in\Prime$ such that $\alpha(p_0)^{r_0} <0$. By Lemma \ref{lem:cla1}, $\alpha(p_0)^r$ is real for any $r$. Hence we have $\alpha(p_0)$ is real and $\alpha(p_0)^{2j}\ge 0$ for any $j\in{\mathbb{N}}$. Thus we can assume that $\alpha(p_0)<0$. On the other hand, One has $a(p)=\alpha(p)<0$ by Remark \ref{rem:an}. Therefore, we have $\widehat{I\!D^0} = \emptyset$ by using Theorem \ref{th:cla1}.
\end{example}

\begin{theorem}\label{th:emp1}
One has $\widehat{I\!D^0} = \emptyset$ if and only if $\eta=1,2$. 
\end{theorem}
\begin{proof}
When $\eta=3$, we have $\widehat{I\!D^0} \ne \emptyset$ from Example \ref{exa:q}. Suppose $\eta \ge 4$, $d= \vc =1$, $\alpha_1(2)=1$, $\alpha_2(2)=-\alpha_3(2)={\rm{i}}$, $\alpha_{4} (2) = \cdots = \alpha_\eta (2) =1/\eta$, and $\alpha_1(p) = \cdots = \alpha_\eta(p) = 0$ for any $p>3$. Then it holds that
\begin{equation*}
\begin{split}
&\sum_{k=1}^\eta \alpha_k (2)^{4j-3}= 1 + \frac{\eta-3}{\eta^{4j-3}}, \qquad 
\sum_{k=1}^\eta \alpha_k (2)^{4j-2}=-1 + \frac{\eta-3}{\eta^{4j-2}}, \\
&\sum_{k=1}^\eta \alpha_k (2)^{4j-1}=1 + \frac{\eta-3}{\eta^{4j-1}}, \qquad 
\sum_{k=1}^\eta \alpha_k (2)^{4j} = 3  + \frac{\eta-3}{\eta^{4j}}
\end{split}
\end{equation*}
for any $j \in {\mathbb{N}}$. Hence there exists $j \in {\mathbb{N}}$ such that $\sum_{k=1}^\eta \alpha_k (2)^{4j-2} < 0$. Moreover, the function
$$
F_\eta (\sigma,t) := \frac{Z_Q(\sigma+{\rm{i}}t)}{Z_Q(\sigma)}
\frac{(1-\eta^{-1} 2^{-\sigma})^{\eta-3}}{(1-\eta^{-1} 2^{-\sigma-{\rm{i}}t})^{\eta-3}} 
$$
is a characteristic function by Example \ref{exa:q} and the fact that the product of a finite number of characteristic functions is also a characteristic function. Thus we have $\widehat{I\!D^0} \ne \emptyset$ when $\eta \ge 3$. Hence we only have to show $\widehat{I\!D^0} \ne \emptyset$ when $\eta=2$ by Example \ref{ex:eta1}.

First suppose $\alpha_1 (p)$ and $\alpha_2 (p)$ are real for all $p \in {\mathbb{P}}$. Then we can see that $\widehat{I\!D^0} \ne \emptyset$ by Lemma \ref{lem:eta2real} below. Next assume $\alpha_1 (p_0) + \alpha_2 (p_0) \in {\mathbb{C}} \setminus {\mathbb{R}}$ for some $p_0 \in {\mathbb{P}}$. Then we have $f_{\vsig} \not \in \widehat{I\!D}$ from Theorem \ref{th:cla1} (I). Furthermore, one has
$$
a(p_0) = \alpha_1 (p_0) + \alpha_2 (p_0) \in {\mathbb{C}} \setminus {\mathbb{R}}.
$$
Thus we also obtain $f_{\vsig} \in \widehat{N\!D}$ in this case. Finally suppose $\alpha_1 (p), \alpha_2 (p) \in {\mathbb{C}} \setminus {\mathbb{R}}$ and $\alpha_1 (p) + \alpha_2 (p) \in {\mathbb{R}}$ for all $p \in {\mathbb{P}}$. Then we can put $\alpha_1(p) := R_p e^{{\rm{i}}\theta_p}$ and $\alpha_2(p) := R_p e^{{-\rm{i}}\theta_p}$, where $R_p, \theta_p>0$. In this case, for each $p \in {\mathbb{P}}$, there exists $r_0 \in {\mathbb{N}}$ such that 
$$
2\alpha_1(p)^{r_0} + 2\alpha_2(p)^{r_0} = 2R_p^{r_0} \cos (r_0\theta_p) <0.
$$
Hence we have $f_{\vsig} \not \in \widehat{I\!D}$ from Theorem \ref{th:cla1} (I). Moreover, there is $j_0 \in {\mathbb{N}}$ which satisfies $\sin (\theta_p) \sin (j_0\theta_p) <0$. Then one has
$$
a(p^{j_0-1}) =\sum_{r=0}^{j_0-1} \alpha_1(p)^r \alpha_2(p)^{j_0-1-r} = 
\frac{\alpha_1(p)^{j_0} - \alpha_2(p)^{j_0}}{\alpha_1(p) - \alpha_2(p)} 
= \frac{R_p^{j_0}\sin (j_0\theta_p)}{R_p \sin (\theta_p)} < 0.
$$
Thus we also obtain $f_{\vsig} \in \widehat{N\!D}$ in this case. Therefore, when $\eta=2$, $f_{\vsig} \not \in \widehat{I\!D}$ implies $f_{\vsig} \in \widehat{N\!D}$.
\end{proof}

\begin{lemma}\label{lem:eta2real}
Let $\eta =2$, $\alpha_1 (p)$ and $\alpha_2 (p)$ be real for all $p \in {\mathbb{P}}$. Then $f_{\vsig} \in \widehat{I\!D}$ if and only if $\alpha_1 (p)+\alpha_2 (p) \ge 0$ for any $p \in {\mathbb{P}}$, and $f_{\vsig} \in \widehat{N\!D}$ if and only if $\alpha_1 (p_0)+\alpha_2 (p_0)<0$ for some $p_0 \in {\mathbb{P}}$. 
\end{lemma}
\begin{proof}
First sup pose $\alpha_1 (p)+\alpha_2 (p) \ge 0$ for any $p \in {\mathbb{P}}$. Then we have $\alpha_1 (p)^r + \alpha_2 (p)^r \ge 0$ for all $r \in {\mathbb{N}}$. It is proved as follows. When $\alpha_1 (p), \alpha_2 (p) \ge 0$, obviously we have $\alpha_1 (p)^r + \alpha_2 (p)^r \ge 0$. Thus we can assume $\alpha_1(p) \ge 0 \ge \alpha_2 (p)$ and $\alpha_1(p) \ge |\alpha_2 (p)|$.  In this case, it holds that
$$
\alpha_1 (p)^r + \alpha_2 (p)^r \ge \alpha_1 (p)^r - \bigl|\alpha_2 (p)^r\bigr| \ge 0.
$$
Next suppose $\alpha_1 (p_0)+\alpha_2 (p_0)<0$ for some $p_0 \in {\mathbb{P}}$. Then we have
$$
a (p_0) = \alpha_1 (p_0)+\alpha_2 (p_0)<0.
$$
Hence we obtain this lemma from (I) and (III) of Theorem \ref{th:cla1}.
\end{proof}

\section{Main results}
Now we define compound Poisson zeta distributions on $\rd$ generated by the multidimensional polynomial Euler products introduced in Section 2. In this section, we consider the multidimensional polynomial Euler product
\begin{equation}
Z_E(\vs) = \prod_p \prod_{l=1}^\varphi \prod_{k=1}^\eta \biggl( 1 - \frac{\alpha_{lk}(p)}{p^{\langle \vc_l,\vs\rangle}} \biggr)^{-1} ,
\label{eq:defpe4}
\end{equation}
where $\min_{1\le l\le \varphi} \Re \langle \vc_l,\vs\rangle >1$, $\alpha_{lk}(p) \in {\mathbb{C}}$ and $|\alpha_{lk}(p)|\le 1$ for $1 \le k \le \eta$ and $1 \le l \le \varphi$. Moreover, we always suppose the following condition (A1) or (A2);
\begin{description}
\item[(A1)] The pair of vectors $\vc_l$, $1 \le l \le \varphi$ are linearly independent.
\item[(A2)] $\vc_l = \gamma_l \vc$, where $1=\gamma_1, \gamma_2, \ldots , \gamma_\varphi$ are algebraic real numbers which are linearly independent over the rationals. 
\end{description}
except for Section 4.4. It is called that real numbers $\theta_1,\ldots,\theta_n$ are linearly independent over the rationals if $\sum_{k=1}^n \gamma_k \theta_k=0$ with rational multipliers $\gamma_1,\ldots,\gamma_n$ implies $\gamma_1 = \cdots = \gamma_n =0$. By Lemma \ref{lem:mdst}, the function $Z_E (\vs)$ is also written by 
\begin{equation}
Z_E (\vs)=
\prod_{l=1}^\varphi \sum_{n_l=1}^{\infty} \frac{a_l(n_l)}{n_l^{\langle \vc_l,\vs\rangle}} =
\sum_{n_1, \ldots ,n_\varphi=1}^\infty
\frac{a_1(n_1)}{n_1^{\langle \vc_1,\vs\rangle}} \cdots \frac{a_\varphi(n_\varphi)}{n_\varphi^{\langle \vc_\varphi,\vs\rangle}},
\label{eq:defsz4}
\end{equation}
where $a_l(n)$ is multiplicative and expressed as (\ref{eq:an}). It should be noted that the multiple series above convergent absolutely when $\min_{1\le l\le \varphi} \Re \langle \vc_l,\vs\rangle >1$ since we have $a_l(n) = O(n^{\varepsilon})$ by Lemma \ref{lem:mdst}. Furthermore, for $\vsig$ satisfying $\min_{1\le l\le \varphi} \Re \langle \vc,\vsig_l \rangle >1$, we define a normalized function $f_{\vsig}\left(\vt\,\right)$ by (\ref{eq:nor1}).

We quote Baker's theorem which is very famous in transcendental number theory. This theorem plays import role in this section (see the proofs of Theorem \ref{th:idm1} and Lemma \ref{lem:m2}). 
\begin{proposition}[see {\cite[Theorem 2.4]{Baker}}]
The numbers $\gamma_1^{\beta_1} \cdots \gamma_n ^{\beta_n}$ are transcendental for any algebraic numbers $\gamma_1 , \ldots , \gamma_n$, other than $0$ or $1$, and any algebraic numbers $\beta_1, \ldots , \beta_n$ with $1,\beta_1, \ldots , \beta_n$ are linearly independent over the rationals.
\label{lem:baker}
\end{proposition}

\subsection{Infinitely divisible or not}
We have the following.
\begin{theorem}\label{th:idm1}
Let $\alpha_{lk}(p) \in {\mathbb{C}}$, $|\alpha_{lk}(p)|\le 1$ for any $p\in\Prime$, $1 \le k \le \eta$, $1 \le l \le \varphi$ in (\ref{eq:defpe4}). Then $f_{\vsig}$ is an infinitely divisible characteristic function if and only if $\sum_{k=1}^\eta \alpha_k(p)^r \ge 0$ for all $r \in {\mathbb{N}}$, $p\in\Prime$ and $1 \le l \le \varphi$.
Moreover, when $\sum_{k=1}^\eta \alpha_{lk}(p)^r \ge 0$ for all $r \in {\mathbb{N}}$, $p\in\Prime$ and $1 \le l \le \varphi$, the normalized function $f_{\vsig}$ is a compound Poisson characteristic function with its finite L\'evy measure $N_{\vsig}^Z$ on $\rd$ given by 
\begin{equation}\label{eq:lmm1}
N_{\vsig}^Z (dx) = \sum_p \sum_{r=1}^{\infty} \sum_{l=1}^\varphi \sum_{k=1}^\eta \frac{1}{r}
\alpha_{lk}(p)^r p^{-r\langle\vc,\vsig\rangle} \delta_{\log p^r \vc_l} (dx).
\end{equation}
\end{theorem}

To prove the theorem above, we show the following lemma. 
\begin{lemma}\label{lem:lmm1}
Let $\alpha_{lk}(p) \in {\mathbb{C}}$, $|\alpha_{lk}(p)|\le 1$ for any $p\in\Prime$, $1 \le k \le \eta$ and $1 \le l \le \varphi$. Then $N_{\vsig}^Z$ is a complex measure on $\rd$ with total variation measure $|N_{\vsig}^Z|$ satisfying $N_{\vsig}^Z(\{0\}) =0$ and $\int_{\rd} (|x| \wedge 1) |N_{\vsig}^Z|(dx) < \infty$.
\end{lemma}

\begin{proof}
Recall that $\log Z_E (\vs)$ is expressed as (\ref{eq:9.19}) and the normalized function $f_{\vsig}(\vt)$ converges absolutely when $\min_{1\le l\le \varphi} \Re \langle \vc_l,\vs\rangle >1$ from Lemma \ref{lem:EPc}. In the view of the proof of Lemma \ref{lem:lm1}, one has
\begin{equation*}
\begin{split}
&\log f_{\vsig}(\vt)= \log \frac{Z_E (\vsig + {\rm i}\vt)}{Z_E (\vsig)} 
= \sum_p \sum_{r=1}^{\infty} \sum_{l=1}^\varphi \sum_{k=1}^\eta \frac{1}{r} 
\alpha_{lk}(p)^r p^{-r\langle\vc_l,\vsig\rangle} \bigl(p^{-r\langle\vc_l,{\rm i}\vt\rangle} -1\bigr) \\ 
= & \sum_p \sum_{r=1}^{\infty} \sum_{l=1}^\varphi \sum_{k=1}^\eta \frac{1}{r} 
\alpha_{lk}(p)^r p^{-r\langle\vc_l,\vsig\rangle} \bigl(e^{-r\langle\vc_l,{\rm i}\vt\rangle \log p} -1\bigr) 
=  \int_{\rd} (e^{-\langle{\rm i}\vt,x\rangle}-1)N_{\vsig}^Z(dx),
\end{split}
\end{equation*}
where $N_{\vsig}^Z$ is defined by (\ref{eq:lmm1}) since we have $e^{-r\langle\vc_l,{\rm i}\vt\rangle \log p}=\int_{\rd} e^{-\langle{\rm i}\vt,x\rangle} \delta_{\log p^r \vc_l} (dx)$. Now put $v := \min_{1\le l\le \varphi} \Re \langle \vc_l,\vs\rangle >1$. By the similar way used in the proof of Lemma \ref{lem:lm1}, we have
\begin{equation*}
\begin{split}
N_{\vsig}^Z (\rd) \le & \int_{\rd}\sum_p \sum_{r=1}^{\infty} \sum_{l=1}^\varphi \sum_{k=1}^\eta \frac{1}{r}
|\alpha_{lk}(p)|^r p^{-r\langle\vc_l,\vsig\rangle} \delta_{\log p^r \vc_l} (dx) \\
= & \, \varphi \eta \sum_p \sum_{r=1}^{\infty} \frac{1}{r} p^{-rv} 
\le 2 \varphi \eta \zeta (v) <\infty.
\end{split}
\end{equation*}
It is also easy to see that the measure $N_{\vsig}^Z$ satisfies $\int_{|x|<1}|x|N_{\vsig}^Z(dx)\le N_{\vsig}^Z(\rd)<\infty$. 
\end{proof}

\begin{proof}[Proof of Thereom \ref{th:idm1}]
First suppose $\sum_{k=1}^\eta \alpha_{lk}(p)^r \ge 0$ for all $r \in {\mathbb{N}}$, $p\in\Prime$ and $1 \le l \le \varphi$. In this case, $N_{\vsig}^Z$ is a measure on $\rd$ with $N_{\vsig}^Z(\{0\}) =0$ and $\int_{\rd} (|x| \wedge 1) N_{\vsig}^Z(dx) < \infty$ by Lemma \ref{lem:lmm1}. Thus $f_{\vsig}$ is an infinitely divisible characteristic function.

Next suppose that there exist $r_0 \in {\mathbb{N}}$, $p_0 \in\Prime$ and $1 \le l_0 \le \varphi$ such that $\sum_{k=1}^\eta \alpha_{l_0k}(p_0)^{r_0} \in K$, where $K := \{ z \in {\mathbb{C}} : |z| \le \eta, z \not \in [0,\eta] \}$. Let $r_1,r_2 \in {\mathbb{N}}$, $p_1, p_2 \in\Prime$ and $1 \le l_1, l_2 \le \varphi$. Now we will show that
\begin{equation}
\log p_1^{r_1} \vc_{l_1} = \log p_2^{r_2} \vc_{l_2} \quad \mbox{ if and only if } \quad 
r_1=r_2, \mbox{ }  p_1=p_2 \mbox{ and } \vc_{l_1} = \vc_{l_2}.
\label{eq:ftam}
\end{equation}
Under the assumption (A1), we have (\ref{eq:ftam}) by the fundamental theorem of arithmetic. Next assume (A2) and let  $\vc_{l_1} = \gamma_l \vc$ and $\vc_{l_2} = \gamma_2 \vc$, where $\vc$ is a non-zero ${\mathbb{R}}^d$-valued vector, and real numbers $\gamma_{l_1}$ and $\gamma_{l_2}$ are linearly independent over the rationals. In this case, we only have to show $\gamma_{l_1} \log p_1^{r_1} = \gamma_{l_2} \log p_2^{r_2}$ if and only if $r_1=r_2$, $p_1=p_2$ and $\gamma_{l_1}=\gamma_{l_2}$. Put $\gamma = \gamma_{l_1}/ \gamma_{l_2}$. Then $\gamma$ is  a non-rational algebraic real number by the assumption. Then $\gamma_{l_1} \log p_1^{r_1} = \gamma_{l_2} \log p_2^{r_2}$ is equivalent to $\gamma \log p_1^{r_1} = \log p_2^{r_2}$, namely $(p_1^{r_1})^\gamma = p_2^{r_2}$. Obvious, $p_2^{r_2}$ is a natural number. On the other hand, $(p_1^{r_1})^\gamma$ is a transcendental number by Proposition \ref{lem:baker}. Thus we have (\ref{eq:ftam}).

Therefore, one has $\delta_{\log {p_1}^{r_1} \vc_{l_1}} (dx) = \delta_{\log {p_2}^{r_2} \vc_{l_2}} (dx)$ if and only if $r_1=r_2$, $p_1=p_2$ and $\vc_{l_1} = \vc_{l_2}$. Hence the normalized function $f_{\vsig}$ is not an infinitely divisible characteristic function by the (not measure but) complex measure $\sum_{k=1}^\eta \alpha_{l_0k}(p_0)^{r_0} \delta_{\log {p_0}^{r_0} \vc_{l_0}}$. 
\end{proof}

\begin{remark}\label{rem:idm1}
If $\gamma$ in the proof above is rational or transcendental, the statement (\ref{eq:ftam}) is not true. For example, when $p_1=p_2=r_1=2$, $r_2=3$ and $\gamma = 3/2$, we have $3 \log p_1^{r_1} = 2 \log p_2^{r_2} = \log 64$. Moreover, when $p_1 =2$, $p_2=3$, $r_1=r_2=1$ and $\gamma = \log 3/ \log 2$, one has $\log 3 \log p_1^{r_1} = \log 2 \log p_2^{r_2} = \log 2 \times \log 3$. 
\end{remark}

\begin{example}\label{ex:idm1}
Let $p_n$ be the $n$-th prime number, $L_{+{\rm{i}}}(s) := \prod_n (1-{\rm{i}}^np_n^{-s})^{-1}$ and $L_{-{\rm{i}}}(s) := \prod_n (1-(-{\rm{i}})^np_n^{-s})^{-1}$ (see Example \ref{ex:id1}).
\pn$(i)$ Functions to be infinitely divisible distributions;\\
$\zeta (s)$, $\,\,\,\zeta ^2(s) L_{+{\rm{i}}}(s) L_{-{\rm{i}}}(s)$, $\,\,\,\zeta (s_1)\zeta ^2(s_2) L_{+{\rm{i}}}(s_2) L_{-{\rm{i}}}(s_2)$, $\,\,\,\zeta (s_1+s_2)\zeta ^2(s_2) L_{+{\rm{i}}}(s_2) L_{-{\rm{i}}}(s_2)$.\\
$(ii)$ Functions not to generate infinitely divisible distributions (actually, these are functions not to generate probability distributions);\\
$L_{+{\rm{i}}}(s)$, $\,\,\,L_{-{\rm{i}}}(s)$, $\,\,\,\zeta ^2(s_1) L_{+{\rm{i}}}(s_2) L_{-{\rm{i}}}(s_2)$, $\,\,\,\zeta ^2(s_1+s_2) L_{+{\rm{i}}}(s_2) L_{-{\rm{i}}}(s_2)$, $\,\,\,\zeta ^2(s_1) L_{+{\rm{i}}}(s_1) L_{-{\rm{i}}}(s_2)$.
\end{example}

\subsection{Distribution or not}
We have the following.
\begin{theorem}\label{th:dm1}
Let $\alpha_{lk}(p) \in {\mathbb{C}}$, $|\alpha_{lk}(p)|\le 1$ for any $p\in\Prime$ $1 \le k \le \eta$ and $1 \le l \le \varphi$ in (\ref{eq:defpe4}). Then $f_{\vsig}$ is a characteristic function if and only if $\prod_{l=1}^\varphi a_l(n_l) \ge 0$ for all $(n_1, \ldots ,n_\varphi) \in {\mathbb{N}}^\varphi$, where $a_l(n_l)$ is defined by (\ref{eq:an}).
\end{theorem}

To show this theorem, we define a multidimensional Shintani zeta random variable $X_{\vsig}$ with probability distribution on ${\mathbb{R}}^d$ given by
\begin{equation}
{\rm Pr} \Bigl(X_{\vsig}= - \bigl( \log n_1 \vc_1 ,\ldots, \log n_\varphi \vc_\varphi \bigr) \Bigr) = 
\frac{1}{Z_E(\vsig)} \prod_{l=1}^\varphi \frac{a_l(n_l)}{n_l^{\langle \vc_l, \vsig \rangle}}
\label{eq:gdsdm}
\end{equation}
when $\prod_{l=1}^\varphi a_l(n_l) \ge 0$ for all $(n_1, \ldots ,n_\varphi) \in {\mathbb{N}}^\varphi$. It is easy to see that these distributions are probability distributions (see also Lemma \ref{lem:m1}) since the right hand side of (\ref{eq:gdsdm}) is not smaller than $0$ by the assumption for $a_l(n_l)$, and 
$$
\prod_{l=1}^\varphi \sum_{n_l=1}^{\infty} \frac{a_l(n_l)n_l^{-\langle \vc_l, \vsig \rangle}}{Z_E(\vsig)} = 
\frac{1}{Z_E(\vsig)}\!\! \sum_{(n_1, \ldots , n_\varphi) \in {\mathbb{N}}} \frac{a_1(n_1)}{n_1^{\langle \vc_1,\vsig \rangle}} \cdots \frac{a_\varphi(n_\varphi)}{n_\varphi^{\langle \vc_\varphi,\vsig \rangle}}
=\frac{Z_E(\vsig)}{Z_E(\vsig)}=1 .
$$

We only have to show the following Lemmas \ref{lem:m1} and \ref{lem:m2} to prove Theorem \ref{th:dm1}. Note that Lemmas \ref{lem:m1} and \ref{lem:m2} are a multi sum version of Lemmas \ref{lem:1} and \ref{lem:2}, respectively. 

\begin{lemma}\label{lem:m1}
Let $X_{\vsig}$ be a Shintani zeta random random variable. Then its characteristic function $f_{\vsig}$ is given by (\ref{eq:nor1}).
\end{lemma}
\begin{proof}
By the definition, we have, for any $\vt \in {\mathbb{R}}^d$, 
\begin{equation*}
\begin{split}
f_{\vsig}(t) = & \prod_{l=1}^\varphi \sum_{n_l=1}^\infty 
e^{{\rm i} \langle \vt, -\log n_l \vc_l \rangle} \frac{a_l(n_l)n_l^{-\langle \vc_l, \vsig \rangle}}{Z_E(\vsig)} =
\frac{1}{Z_E(\vsig)} \prod_{l=1}^\varphi \sum_{n_l=1}^\infty 
\frac{e^{-{\rm i} \langle \vc_l, \vt \rangle \log n_l}a_l(n_l)}{n_l^{\langle \vc_l, \vsig \rangle}} \\ = &\frac{1}{Z_E(\vsig)} \prod_{l=1}^\varphi \sum_{n_l=1}^\infty  
\frac{a_l(n_l)}{n_l^{\langle \vc_l, \vsig \rangle + {\rm i} \langle \vc_l, \vt \rangle}}  =
\frac{1}{Z_E(\vsig)}\!\! \sum_{(n_1, \ldots , n_\varphi) \in {\mathbb{N}}} \frac{a_1(n_1)}{n_1^{\langle \vc_1,\vsig \rangle + {\rm i} \langle \vc_1, \vt \rangle}} \cdots \frac{a_\varphi(n_\varphi)}{n_\varphi^{\langle \vc_\varphi,\vsig \rangle + {\rm i} \langle \vc_\varphi, \vt \rangle}} \\ = &
\frac{Z_E(\vsig +{\rm i}\vt)}{Z_E(\vsig)}.
\end{split}
\end{equation*}
This equality implies the lemma.
\end{proof}

\begin{lemma}\label{lem:m2}
Suppose that there exists a pair of integers $(m_1, \ldots, m_\varphi) \in {\mathbb{N}}^\varphi$ such that $\prod_{l=1}^\varphi a_l(m_l) \in {\mathbb{C}} \setminus {\mathbb{R}}_{\ge 0}$. Then  $Z_E(\vsig +{\rm i}\vt)/Z_E(\vsig)$ is not a characteristic function. 
\end{lemma}
\begin{proof}
Let ${\mathbb{N}}_+^\varphi$ be the set of pairs of integers $(n_1, \ldots , n_\varphi)$ such that $\prod_{l=1}^\varphi a_l(n_l)\ge 0$ and ${\mathbb{N}}_+^{\varphi c}$ be the set of pairs of integers $(m_1 ,\ldots, m_\varphi)$ such that $\prod_{l=1}^\varphi a_l(m_l) \in {\mathbb{C}} \setminus {\mathbb{R}}_{\ge 0}$. From the view point of (\ref{eq:gdsdm}), it holds that 
\begin{equation*}
\begin{split}
& \frac{Z_E(\vsig +{\rm i}\vt)}{Z_E(\vsig)} = \frac{1}{Z_E(\vsig)} \prod_{l=1}^\varphi \sum_{n_l=1}^\infty 
\frac{a_l(n_l)}{n_l^{\langle \vc, \vs \rangle}} =
\frac{1}{Z_E(\vsig)} \sum_{n_1, \ldots ,n_\varphi=1}^\infty \frac{a_1(n_1)}{n_1^{\langle \vc_1,\vs\rangle}}
\cdots \frac{a_\varphi(n_\varphi)}{n_\varphi^{\langle \vc_\varphi,\vs\rangle}} \\ = &
\frac{1}{Z_E(\vsig)}\!\! \sum_{(n_1, \ldots , n_\varphi) \in {\mathbb{N}}_+^\varphi} \frac{a_1(n_1)}{n_1^{\langle \vc_1,\vs\rangle}} \cdots \frac{a_\varphi(n_\varphi)}{n_\varphi^{\langle \vc_\varphi,\vs\rangle}} \,+\,
\frac{1}{Z_E(\vsig)}\!\! \sum_{(m_1, \ldots , m_\varphi) \in {\mathbb{N}}_+^{\varphi c}} \frac{a_1(m_1)}{m_1^{\langle \vc_1,\vs\rangle}} \cdots \frac{a_\varphi(m_\varphi)}{m_\varphi^{\langle \vc_\varphi,\vs\rangle}} \\ = &
\frac{1}{Z_E(\vsig)}\!\! \sum_{(n_1, \ldots , n_\varphi) \in {\mathbb{N}}_+^\varphi} \frac{a_1(n_1)}{n_1^{\langle \vc_1,\vsig \rangle}} \cdots \frac{a_\varphi(n_\varphi)}{n_\varphi^{\langle \vc_\varphi,\vsig \rangle}}
\int_{\mathbb{R}^d} e^{{\rm i}\langle \vt, x \rangle} \delta_{- \sum_{l=1}^\varphi \log n_l \vc_l} (dx) \\ &+
\frac{1}{Z_E(\vsig)}\!\! \sum_{(m_1, \ldots , m_\varphi) \in {\mathbb{N}}_+^{\varphi c}} \frac{a_1(m_1)}{m_1^{\langle \vc_1,\vsig\rangle}} \cdots \frac{a_\varphi(m_\varphi)}{m_\varphi^{\langle \vc_\varphi,\vsig\rangle}}
\int_{\mathbb{R}^d} e^{{\rm i}\langle \vt, x \rangle} \delta_{- \sum_{l=1}^\varphi \log m_l \vc_l} (dx) .
\end{split}
\end{equation*}

We have to check
\begin{equation}
\sum_{l=1}^\varphi \log n_l \vc_l = \sum_{l=1}^\varphi \log m_l \vc_l \quad \mbox{ if and only if } \quad 
n_l=m_l \mbox{ for all } 1 \le l \le \varphi.
\label{eq:logli1}
\end{equation}
Under the assumption (A1), the statement above is obvious. Thus suppose (A2). In this case, the assertion (\ref{eq:logli1}) is equivalent to
\begin{equation*}
\begin{split}
&\log n_1 + \gamma_2 \log n_2 + \cdots + \gamma_\varphi \log n_\varphi = 
\log m_1 + \gamma_2 \log m_2 + \cdots + \gamma_\varphi \log m_\varphi \\
&\mbox{if and only if } \quad n_l=m_l \,\, \mbox{ for all } \,\, 1 \le l \le \varphi.
\end{split}
\end{equation*}
Though this is immediately proved by \cite[Proposition 2.2]{Nakamurasr1}, we write the proof here for convenience of readers. Suppose $n_1/m_1 = (m_2/n_2)^{\gamma_2} \cdots (m_\varphi/n_\varphi)^{\gamma_\varphi}$. Obviously, $n_1/m_1$ is a rational number. If there exists $2 \le l \le \varphi$ such that $n_l \ne m_l$, the number $(m_2/n_2)^{\gamma_2} \cdots (m_\varphi/n_\varphi)^{\gamma_\varphi}$ is transcendental by Proposition \ref{lem:baker}. Thus we have $n_l=m_l$ for any $2 \le l \le \varphi$. Then one has $n_1/m_1=1$. Hence we obtain (\ref{eq:logli1}). 

By Remark \ref{rem:an}, we have $a_l(1)=1$ for any $1 \le l \le \varphi$. Hence 
\begin{equation*}
\begin{split}
&\frac{1}{Z_E(\vsig)}\!\! \sum_{(n_1, \ldots , n_\varphi) \in {\mathbb{N}}_+^\varphi} \frac{a_1(n_1)}{n_1^{\langle \vc_1,\vsig \rangle}} \cdots \frac{a_\varphi(n_\varphi)}{n_\varphi^{\langle \vc_\varphi,\vsig \rangle}}
 \delta_{- \sum_{l=1}^\varphi \log n_l \vc_l} (dx) \\ &+
\frac{1}{Z_E(\vsig)}\!\! \sum_{(m_1, \ldots , m_\varphi) \in {\mathbb{N}}_-^\varphi} \frac{a_1(m_1)}{m_1^{\langle \vc_1,\vsig\rangle}} \cdots \frac{a_\varphi(m_\varphi)}{m_\varphi^{\langle \vc_\varphi,\vsig\rangle}}
\delta_{- \sum_{l=1}^\varphi \log m_l \vc_l} (dx)
\end{split}
\end{equation*}
is not a measure but a complex signed measure with finite total variation by the coefficients $\prod_{l=1}^\varphi a_l(1) >0$ and $\prod_{l=1}^\varphi a_l(m_l) \in {\mathbb{C}} \setminus {\mathbb{R}}_{\ge 0}$, Lemma \ref{lem:mdst} and 
$$
\int_{\mathbb{R}^d} \prod_{l=1}^\varphi \sum_{n_l=1}^\infty \biggl| 
\frac{a_l(n_l)}{n_l^{\langle \vc, \vsig \rangle}} \biggr| \delta_{- \sum_{l=1}^\varphi \log m_l \vc_l} (dx) =
\prod_{l=1}^\varphi \sum_{n_l=1}^\infty \frac{|a_l(n_l)|}{n_l^{\langle \vc, \vsig \rangle}} < \infty . 
$$
Therefore $Z_E(\vsig +{\rm i}\vt)/Z_E(\vsig)$ is not a characteristic function. 
\end{proof}

\begin{remark}\label{rem:dnm1}
If where $1=\gamma_1, \gamma_2, \ldots ,\gamma_\varphi$ are linearly dependent over the rationals, the statement (\ref{eq:logli1}) is not true. We have the same counter example treated in Remark \ref{rem:idm1}. 
\end{remark}

\begin{example}
We use the same notation appeared in Example \ref{exa:dqi1}. 
\pn$(iii)$ Functions to be probability distributions (actually, these are functions to generate infinitely divisible distributions);\\
$\zeta (s)$, $\,\,\,\zeta (s) L(s)$, $\,\,\,\zeta (s_1) L(s_1) \zeta(s_2)$, $\,\,\,\zeta (s_1+s_2)\zeta (s_2) L(s_2)$, $\,\,\,\zeta (s_1+s_2) L(s_1+s_2) \zeta(s_2) L(s_2)$.
\pn$(iv)$ Functions not to generate probability distributions;\\
$L(s)$, $\,\,\,\zeta (s_1) L(s_2)$, $\,\,\,\zeta (s_1) L(s_1) L(s_2)$, $\,\,\,\zeta (s_1+s_2) L(s_2) L(s_2)$, $\,\,\,\zeta (s_1+s_2) \zeta(s_2) L(s_2) L(s_2)$.
\end{example}

\subsection{Classification}
Let $F^Z_E$ be the set of normalized functions of $Z_E(\vsig +{\rm i}\vt)/Z_E(\vsig)$, where $Z_E(\vsig +{\rm i}\vt)$ is defined by (\ref{eq:defpe4}) or (\ref{eq:defsz4}). We use the same notation $\widehat{I\!D}$, $\widehat{I\!D^0}$ and $\widehat{N\!D}$ defined at the beginning of Section 3.3. Then we have the following theorem.
\begin{theorem}\label{th:clam1}
Let $f_{\vsig} \in F^Z_E$, $\alpha_{lk}(p) \in {\mathbb{C}}$, $|\alpha_{lk}(p)|\le 1$ for any $p\in\Prime$, $1 \le k \le \eta$ and $1 \le l \le \varphi$ in (\ref{eq:defpe4}). Then we have\\
$(I)$ $f_{\vsig} \in \widehat{I\!D}$ if and only if $\sum_{k=1}^\eta \alpha_{lk}(p)^r \ge 0$ for all $r \in {\mathbb{N}}$, $p\in\Prime$ and $1 \le l \le \varphi$.\\
$(II)$ $f_{\vsig} \in \widehat{I\!D^0}$ if and only if $\prod_{l=1}^\varphi a_l(n_l) \ge 0$ for all $(n_1, \ldots , n_\varphi) \in {\mathbb{N}}^\varphi$, and there exist $r_0 \in {\mathbb{N}}$, $p_0 \in\Prime$ and $1 \le l_0 \le \varphi$ such that $\sum_{k=1}^\eta \alpha_{l_0k}(p_0)^{r_0} <0$.\\
$(III)$ $f_{\vsig} \in \widehat{N\!D}$ if and only if there exists a pair of integers $(m_1, \ldots, m_\varphi) \in {\mathbb{N}}^\varphi$ such that $\prod_{l=1}^\varphi a_l(m_l)  \in {\mathbb{C}} \setminus {\mathbb{R}}_{\ge 0}$. \\
$(IV)$ $F^Z_E= \widehat{I\!D} \biguplus \widehat{I\!D^0} \biguplus \widehat{N\!D}$. 
\end{theorem}

To show the theorem above, we only have to prove the following lemma the same as in Section 3.3.
\begin{lemma}\label{lem:clam1}
Let $\alpha_{lk}(p) \in {\mathbb{C}}$, $|\alpha_{lk}(p)|\le 1$ for any $p\in\Prime$, $1 \le k \le \eta$ and $1 \le l \le \varphi$ in (\ref{eq:defpe4}). Suppose that $\prod_{l=1}^\varphi a_l(n_l) \ge 0$ for all $(n_1, \ldots , n_\varphi) \in {\mathbb{N}}^\varphi$, where $a_l(n_l)$ is defined by (\ref{eq:an}). Then $\sum_{k=1}^\eta \alpha_{lk}(p)^r$ is real for any $r \in {\mathbb{N}}$, $p\in\Prime$ and $1 \le l \le \varphi$. 
\end{lemma}
\begin{proof}
By (\ref{eq:defsr3}) and the assumption that $\prod_{l=1}^\varphi a_l(n_l) \ge 0$ for all $(n_1, \ldots , n_\varphi) \in {\mathbb{N}}^\varphi$, one has $\overline{Z_E(\vsig+\vt)} = Z_E (\vsig-\vt)$. From the view of the proof of Theorem \ref{th:idm1}, we have
\begin{equation*}
\begin{split}
\log \overline{f_\vsig (\vt)} =& \sum_p \sum_{r=1}^{\infty} \sum_{l=1}^\varphi \sum_{k=1}^\eta \frac{1}{r} 
\overline{\alpha_{lk}(p)^r} p^{-r\langle\vc_l,\vsig\rangle} \bigl(p^{r\langle\vc_l,{\rm i}\vt\rangle} -1\bigr), \\
\log f_\vsig(-\vt) =& \sum_p \sum_{r=1}^{\infty} \sum_{l=1}^\varphi \sum_{k=1}^\eta \frac{1}{r} 
\alpha_{lk}(p)^r p^{-r\langle\vc_l,\vsig\rangle} \bigl(p^{r\langle\vc_l,{\rm i}\vt\rangle} -1\bigr).
\end{split}
\end{equation*}
Therefore, we obtain $\sum_{l=1}^\varphi \sum_{k=1}^\eta \overline{\alpha_{lk}(p)} {}^r = \sum_{l=1}^\varphi \sum_{k=1}^\eta \alpha_{lk}(p)^r$ for any $r \in {\mathbb{N}}$, $p\in\Prime$ and $1 \le l \le \varphi$. 
\end{proof}

\begin{example}
Put $Z_Q (s) := (1-2^{-s})^{-1}(1+2^{-2s})^{-1}$ (see Example \ref{exa:q} and Remark \ref{re:col}). Then the functions $\zeta (s_1) Z_Q(s_2)$, $\zeta(s_1+s_2) Z_Q(s_1)$, $Z_Q(s_1) Z_Q(s_2)$ and $Z_Q(s_1+s_2) Z_Q(s_1)$ generate quasi-infinite divisible but non-infinite divisible characteristic functions. 
\end{example}

Therefore we generally have $\widehat{I\!D^0} \ne \emptyset$ (see also Example \ref{exa:q}). The following proposition implies $\widehat{I\!D^0} = \emptyset$ when $\alpha_{lk}(p) \in \{ 0,1,-1\}$. This result is proved in \cite[Theorem 3.15]{ANPE}. We give a simpler proof here by using Theorem \ref{th:clam1}. 
\begin{proposition}[{\cite[Theorem 3.15]{ANPE}}]
Let $\alpha_{lk}(p) \in \{ 0,1,-1\}$ for any $p\in\Prime$, $1 \le k \le \eta$ and $1 \le l \le \varphi$ in (\ref{eq:defpe4}). Then it holds that $F_E^Z= \widehat{I\!D} \biguplus \widehat{N\!D}$. Namely, we have $f_{\vsig} \in \widehat{N\!D}$ if and only if there exist $r_0 \in {\mathbb{N}}$, $p_0 \in\Prime$ and $1 \le l_0 \le \varphi$ such that $\sum_{k=1}^\eta \alpha_{l_0k}(p_0)^{r_0} <0$. 
\end{proposition}
\begin{proof}
We only have to show $\widehat{I\!D^0} = \emptyset$ in this case. Suppose that $f_{\vsig} \not \in \widehat{I\!D}$, namely there exist $r_0 \in {\mathbb{N}}$,  $p_0 \in\Prime$ and $1 \le l_0 \le \varphi$ such that $\sum_{k=1}^\eta \alpha_{l_0k}(p_0)^{r_0} <0$. By the assumption $\alpha_{lk}(p) \in \{ 0,1,-1\}$, we have $\sum_{k=1}^\eta \alpha_{l_0k}(p_0)^{2j} \ge 0$ and $\sum_{k=1}^\eta \alpha_{l_0k}(p_0) = \sum_{k=1}^\eta \alpha_{l_0k}(p_0)^{2j-1}$ for any $j \in {\mathbb{N}}$. Thus we can assume that $\sum_{k=1}^\eta \alpha_{l_0k}(p_0) <0$. On the other hand, we have $a_{l}(p) = \sum_{k=1}^\eta \alpha_{lk}(p)$ for any $1 \le l \le \varphi$ and $p \in {\mathbb{P}}$ by Remark \ref{rem:an}. Hence, from Theorem \ref{th:dm1}, we obtain $f_{\vsig} \in \widehat{N\!D}$ since we have $a_{l_0}(p_0) \prod_{l\ne l_0}^\varphi a_l(1) <0$ when $\sum_{k=1}^\eta \alpha_{l_0k}(p_0) <0$. Therefore, one has $f_{\vsig} \in \widehat{N\!D}$ when $f_{\vsig} \not \in \widehat{I\!D}$, equivalently $f_{\vsig} \in \widehat{I\!D^0} \biguplus \widehat{N\!D}$, under the assumption $\alpha_{lk}(p) \in \{ 0,1,-1\}$. 
\end{proof}

\begin{example}\label{ex:mtheta1}
When $\eta=1$, one has $\widehat{I\!D^0} = \emptyset$. This is proved by a method similar to the one used in the proof of Example \ref{ex:eta1}. Suppose there exist $1 \le l_0 \le \varphi$, $r_0 \in {\mathbb{N}}$ and $p_0 \in\Prime$ such that $\alpha(p_0)^{r_0} <0$. By Lemma \ref{lem:clam1}, $\alpha_l(p_0)^r$ is real for any $l$ and $r$. Thus we can assume that $\alpha_{l_0}(p_0)<0$. One has $a_l(p)=\alpha_l(p)<0$ by Remark \ref{rem:an}. Therefore, we have $\widehat{I\!D^0} = \emptyset$ by using Theorem \ref{th:clam1}.
\end{example}

\begin{theorem}\label{th:emp2}
One has $\widehat{I\!D^0} = \emptyset$ if and only if $\eta=1,2$. 
\end{theorem}
\begin{proof}
We only have to show $\widehat{I\!D^0} = \emptyset$ when $\eta=2$ from Example \ref{ex:mtheta1} and the first part of  the proof of Theorem \ref{th:emp1} which implies $\widehat{I\!D^0} \ne \emptyset$ if $\eta \ge 3$. Assume that $f_{\vsig} \not \in \widehat{I\!D}$, namely, there exist $j_0 \in {\mathbb{N}}$, $p_0 \in \Prime$ and $1 \le l_0 \le \varphi$ such that $\sum_{k=1}^2 \alpha_{l_0k}(p_0)^{j_0} \in {\mathbb{C}} \setminus {\mathbb{R}}_{\ge 0}$. From the proof of Lemma \ref{lem:eta2real}, this is equivalent to that there are $p_0 \in \Prime$ and $1 \le l_0 \le \varphi$ satisfying $\sum_{k=1}^2 \alpha_{l_0k}(p_0)$ is negative or non-real. Then we can show that there exists $r_0 \in {\mathbb{N}}$ which satisfies $a_{l_0}(p_0^{r_0}) \in {\mathbb{C}} \setminus {\mathbb{R}}_{\ge 0}$ by modifying the proof of Theorem \ref{th:emp1}. Hence we have $a_{l_0}(p_0^{r_0}) \prod_{l\ne l_0}^\varphi a_l(1)$ is negative or non-real since $a_l(1)=1$ for all $1 \le l \le \varphi$. Therefore, we obtain $f_{\vsig} \in \widehat{N\!D}$ from Theorem \ref{th:clam1} (III). 
\end{proof}

\subsection{Some linearly dependent cases}
In this subsection, we consider the case $a_k \vc_{k_1} = b _k \vc_{k_2}$, where $k_1 \ne k_2$ and $a_k$ and $b_k$ are some positive integers. For simplicity, we treat the case $\varphi =2$, $\vc_1 = \gamma_1\vc$ and $\vc_2 = \gamma_2\vc$, where $\vc$ is a non-zero ${\mathbb{R}}^d$-valued vector and $\gamma_1,\gamma_2$ are natural numbers. Let $\alpha_l(p) \in {\mathbb{C}}$ and $|\alpha_l(p)| \le 1$ for $l=1,2$ and $p \in {\mathbb{P}}$. Then we have
\begin{equation}\label{eq:bunkai}
\bigl(1-\alpha_l (p) p^{-\langle \vc_l, \vs \rangle}\bigr) = \prod_{k=1}^{\gamma_l} 
\bigl(1-\alpha_l (p)^{1/\gamma_l} \exp (2 \pi {\rm{i}} k/\gamma_l) p^{-\langle \vc, \vs \rangle}\bigr),
\end{equation}
for $l=1,2$ and $p \in {\mathbb{P}}$. Therefore it holds that
\begin{equation*}
\begin{split}
\prod_p \prod_{l=1}^2 \bigl(1-\alpha_l (p) p^{-\langle \vc_l, \vs \rangle}\bigr)^{-1} =
\prod_p \prod_{l=1}^2 \prod_{k_l=1}^{\gamma_l} 
\bigl(1-\alpha_l (p)^{1/\gamma_l} \exp (2 \pi {\rm{i}} k_l/\gamma_l) p^{-\langle \vc, \vs \rangle}\bigr)^{-1}.
\end{split}
\end{equation*}
The right-hand side of the formula above is written by (\ref{eq:defpe3}) since $|\exp (2 \pi {\rm{i}} k_l/\gamma_l)|=1$. Hence, the case $a \vc_1 = b \vc_2$, where $a$ and $b$ are some positive integers, is reduce to the case treated in Section 3. 

The following examples (i) and (ii) are appeared in \cite[Example 4.2 (ii)]{ANPE}. We give a simple proof here for (i) and (ii). Example (iii) is a completely new one. 
\begin{example}\label{exa:123}
Let $L(s)$ be the function given in (\ref{eq:defLep}) or (\ref{eq:defLds}). For $\sigma >1$, \\
$(i)$ $\zeta (s)^2 L(2s)$ generates an infinitely divisible characteristic function,\\
$(ii)$ $L(s) \zeta (2s)$ does not generate even a characteristic function,\\
$(iii)$ $\zeta (s) L(2s)$ generates a quasi-infinitely divisible characteristic function.
\end{example}

\begin{proof}
$(i)$ From the view of (\ref{eq:bunkai}), we have
$$
\zeta (s)^2 L(2s) = \prod_p (1-p^{-s})^{-2} \times \prod_{p \,:\, {\rm{odd}}} 
\Bigl( 1-(-1)^{\frac{p-1}{4}}p^{-s} \Bigr)^{-1} \Bigl( 1+(-1)^{\frac{p-1}{4}}p^{-s} \Bigr)^{-1} .
$$
Thus we can take $\eta=4$, $\vc =1$, $\alpha_1(2)=\alpha_2(2)=1$, $\alpha_3(2)=\alpha_4(2)=0$, $\alpha_1(p)=\alpha_2(p)=1$, $\alpha_3(p)= (-1)^{\frac{p-1}{4}}$ and $\alpha_4(p)= -(-1)^{\frac{p-1}{4}}$, for $p>3$ in (\ref{eq:defpe3}). Then we have
\begin{equation*}
\begin{split}
&\alpha_1 (p)^{2j-1} + \alpha_2 (p)^{2j-1} + \alpha_3 (p)^{2j-1} + \alpha_4 (p)^{2j-1}
= 2+ (-1)^{(2j-1)\frac{p-1}{4}}- (-1)^{(2j-1)\frac{p-1}{4}} =2 ,\\
&\alpha_1 (p)^{2j} + \alpha_2 (p)^{2j} + \alpha_3 (p)^{2j}  + \alpha_4 (p)^{2j}
= 2+ (-1)^{j\frac{p-1}{2}} + (-1)^{j\frac{p-1}{2}} = 0 \mbox{ or } 4 .
\end{split}
\end{equation*}
Hence $\zeta (s)^2 L(2s)$ generates an infinitely divisible characteristic function by Theorem \ref{th:cla1}. 

$(ii)$ By using  (\ref{eq:bunkai}), we have
$$
\zeta (2s) L(s) = \prod_p (1-p^{-s})^{-1}  (1+p^{-s})^{-1} \times \prod_{p \,:\, {\rm{odd}}} 
\Bigl( 1-(-1)^{\frac{p-1}{2}}p^{-s} \Bigr)^{-1} .
$$
Hence we can take $\eta=3$, $\vc =1$, $\alpha_1(2)=-\alpha_2(2)=1$, $\alpha_3(2)=0$, $\alpha_1(p)=-\alpha_2(p)=1$, and $\alpha_3(p)= (-1)^{\frac{p-1}{2}}$, for $p>3$. In this case one has
$$
a(p) = \alpha_1(p)+\alpha_2(p)+\alpha_3(p) = 1 -1 + (-1)^{\frac{p-1}{2}} = (-1)^{\frac{p-1}{2}}.
$$
When $p \equiv 3 \mod 4$, we have $(-1)^{\frac{p-1}{2}}=-1$. Therefore $L(s) \zeta (2s)$ does not generate even a characteristic function by (III) of Theorem \ref{th:cla1}.

$(iii)$ From (\ref{eq:bunkai}), it holds that
$$
\zeta (s) L(2s) = \prod_p (1-p^{-s})^{-1} \times \prod_{p \,:\, {\rm{odd}}} 
\Bigl( 1-(-1)^{\frac{p-1}{4}}p^{-s} \Bigr)^{-1} \Bigl( 1+(-1)^{\frac{p-1}{4}}p^{-s} \Bigr)^{-1} .
$$
Thus we can take $\eta=3$, $\vc =1$, $\alpha_1(2)=1$, $\alpha_2(2)=\alpha_3(2)=0$, $\alpha_1(p)=1$, $\alpha_2(p)= (-1)^{\frac{p-1}{4}}$ and $\alpha_3 (p)= -(-1)^{\frac{p-1}{4}}$, for $p>3$ in (\ref{eq:defpe3}). Then we have
\begin{equation*}
\begin{split}
&\alpha_1 (p)^{2j-1} + \alpha_2 (p)^{2j-1} + \alpha_3 (p)^{2j-1} 
= 1+ (-1)^{(2j-1)\frac{p-1}{4}}- (-1)^{(2j-1)\frac{p-1}{4}} =1 ,\\
&\alpha_1 (p)^{2j} + \alpha_2 (p)^{2j} + \alpha_3 (p)^{2j} 
= 1+ (-1)^{j\frac{p-1}{2}} + (-1)^{j\frac{p-1}{2}} = -1 \mbox{ or } 3 .
\end{split}
\end{equation*}
Especially, one has $\alpha_1 (p)^{2j} + \alpha_2 (p)^{2j} + \alpha_3 (p)^{2j}=-1$ when $p \equiv 3 \mod 4$ and $j$ is odd. On the other hand, we have
$$
\zeta (s) L(2s) = \prod_p \frac{1+p^{-s}}{1-p^{-2s}} \times \prod_{p \,:\, {\rm{odd}}} 
\Bigl( 1-(-1)^{\frac{p-1}{2}}p^{-2s} \Bigr)^{-1} = \zeta (2s) L(2s) \prod_p (1+p^{-s}). 
$$
It is well-known that (see for example \cite[(1.2.7)]{Tit})
$$
\prod_p (1+p^{-s}) = \prod_p \frac{1-p^{-2s}}{1-p^{-s}} =\frac{\zeta (s)}{\zeta (2s)} = 
\sum_{n=1}^\infty \frac{|\mu (n)|}{n^s},
$$
where $\mu (1)=1$, $\mu (n) =(-1)^j$ if $n$ is the product of $j$ different primes, and $\mu (n)=0$ if $n$ contains any factor to a power higher than the first. Thus we have
$$
\zeta (s) L(2s) = \sum_{n_1=1}^\infty \frac{|\mu (n_1)|}{n_1^s} \sum_{n_2=1}^\infty \frac{a(n_2)}{n_2^s}, \qquad
a(n_2) := 
\begin{cases}
a^\# (n_2) & n_2 = j^2, \quad j=1,2,3,\ldots,\\
0 & \mbox{otherwise},
\end{cases}
$$
where $a^\# (n)$ is defined by (\ref{eq:ash}). By well known fact relates products of Dirichlet series with the Dirichlet convolution of their coefficients (see for example \cite[Theorem 11.5]{Apo}), it holds that
$$
\zeta (s) L(2s) = \sum_{n=1}^\infty \frac{A(n)}{n^s}, \qquad A(n) := \sum_{m|n} |\mu (m)| a(n/m) .
$$
Obviously, we have $A(n)\ge 0$ since $|\mu(n)|\ge 0$ and $a^\# (n) \ge 0$ for any $n \in {\mathbb{N}}$. Therefore, $\zeta (s) L(2s)$ generates a quasi-infinitely divisible but not infinitely divisible characteristic function by (II) of Theorem \ref{th:cla1}. 
\end{proof}

Let $p$ be a prime number, $\sigma_1, \sigma_2 >0$,
\begin{equation*}
\begin{split}
&g_p^\# (\vsig,\vt) := 
\bigl( 1-p^{-(\sigma_1+{\rm{i}}t_1)} \bigr)^{-1} \bigl( 1-p^{-(\sigma_2+{\rm{i}}t_2)} \bigr)^{-1}, 
\qquad g_p^\# (\sigma,t) := \bigl( 1-p^{-(\sigma+{\rm{i}}t)} \bigr)^{-2}, \\ 
&g_p^* (\vsig,\vt) := \bigl( 1+p^{-(\sigma_1+{\rm{i}}t_1)-(\sigma_2+{\rm{i}}t_2)} \bigr)^{-1},
\qquad g_p^* (\sigma,t) := \bigl( 1+p^{-2(\sigma+{\rm{i}}t)} \bigr)^{-1}, 
\end{split}
\end{equation*}
\begin{equation*}
\begin{split}
&G_p^\# (\vsig,\vt) := g_p^\# (\vsig,\vt) / g_p^\# (\vsig,{\Vec{0}}), \qquad
G_p^* (\vsig,\vt) := g_p^* (\vsig,\vt) / g_p^* (\vsig,{\Vec{0}}), \\
&G_p^\# (\sigma,t) := g_p^\# (\sigma,t) / g_p^\# (\sigma,0), \qquad \,
G_p^* (\sigma,t) := g_p^* (\sigma,t) / g_p^* (\sigma,0).
\end{split}
\end{equation*}
We can regard $G_p^\# (\sigma,t)$ and $G_p^* (\sigma,t)$ as $G_p^\# (\vsig,\vt)$ and $G_p^* (\vsig,\vt)$ with $\sigma :=\sigma_1 = \sigma_2$ and $t:= t_1=t_2$. We have to emphasize that $G_p^\#G_p^* (\vsig,\vt) \in \widehat{I\!D^0}({\mathbb{R}}^2)$ but $G_p^\# G_p^*(\sigma,t) \in \widehat{I\!D} ({\mathbb{R}})$ in the following example. Therefore, we can say that the linearly dependent case and linearly independent case are completely different in this example. 
\begin{example}\label{ex:FE}
We have the following.\\
$(iii)$ $G_p^\# (\vsig,\vt) \in \widehat{I\!D} ({\mathbb{R}}^2)$, $G_p^* (\vsig,\vt) \in \widehat{N\!D}({\mathbb{R}}^2)$ and $G_p^\#G_p^* (\vsig,\vt) \in \widehat{I\!D^0}({\mathbb{R}}^2)$.\\
$(iv)$ $G_p^\# (\sigma,t) \in \widehat{I\!D} ({\mathbb{R}})$, $G_p^* (\sigma,t) \in \widehat{N\!D} ({\mathbb{R}})$ and $G_p^\# G_p^*(\sigma,t) \in \widehat{I\!D} ({\mathbb{R}})$. 
\end{example}
\begin{proof}
The example $(iii)$ coincides with \cite[Theorem 2.1]{AN12q}. Thus we only have to show $(iv)$. One has $G_p^\# (\sigma,t) \in \widehat{I\!D} ({\mathbb{R}})$ by Theorem \ref{th:id1}. We can see that $G_p^* (\sigma,t) \in \widehat{N\!D} ({\mathbb{R}})$ by Theorem \ref{th:cla1} $(III)$ and $(1+p^{2(\sigma+{\rm{i}}t)})^{-1} = \sum_{n=0}^\infty (-1)^n p^{-2(\sigma+{\rm{i}}t)}$. We can obtain $G_p^\# G_p^*(\sigma,t) \in \widehat{I\!D} ({\mathbb{R}})$ by $1+p^{2(\sigma+{\rm{i}}t)} = (1+{\rm{i}}p^{\sigma+{\rm{i}}t})(1-{\rm{i}}p^{\sigma+{\rm{i}}t})$ and Theorem \ref{th:id1} (see also Example \ref{ex:id1} $(ii)$). 
\end{proof}

\section{Applications to analytic number theory}
As applications of Multidimensional $\eta$-tuple $\varphi$-rank compound Poisson zeta distributions on ${\mathbb{R}}^d$ to analytic number theory, we investigate the value-distribution of zeta and $L$-functions (see \cite{AN12q}, \cite{AY} and \cite{Nakamura12} for applications to probability theory). 

In the second decade of the twentieth century, Harald Bohr studied on the value-distribution of the Riemann zeta function by applying diophantine, geometric, and probabilistic methods. In recent years, there are many researches on the value-distribution of zeta functions in probabilistic view (see for example Laurin\v cikas \cite{Lau}, Matsumoto \cite{Mas}, Steuding \cite{Steuding1}). In these studies, the Riemann zeta function and Dirichlet $L$-functions are  treated similarly. However, from Theorem \ref{th:cla1}, the Riemann zeta function $\zeta (s)$ can generates a characteristic function but the Dirichlet $L$-function $L(s)$ defined by (\ref{eq:defLep}) or (\ref{eq:defLds}) can not. In this section, we show that the value-distribution of these zeta and $L$-functions in the regions of absolute convergence are different by the facts above. 

\subsection{Almost periodicity and self-approximation}
Bohr \cite{Bohr} proved that every Dirichlet series $f(s)$, having a finite abscissa of absolute convergence $\sigma_a$ is almost periodic in the half-plane $\sigma > \sigma_a$; i.e., for any given $\delta$ and $\varepsilon$, there exists a length $l:= l(f,\delta,\varepsilon)$ such that every interval of length $l$ contains a number $\tau$ for which
$$
\bigl| f(\sigma + {\rm{i}}t + {\rm{i}}\tau) - f(\sigma + {\rm{i}}t) \bigr| < \varepsilon
$$
holds for any $\sigma \ge \sigma_a+\delta$ and for all $t$. Moreover, Bohr showed if $\chi$ is non-principal, then the Riemann hypothesis for Dirichlet $L$-function $L(s,\chi)$ is equivalent to the almost periodicity in the sense of Bohr of $L(s,\chi)$ in $\sigma >1/2$. Recently, Girondo and Steuding \cite{GiSt} gave effective bounds for these lengths in the case of polynomial Euler products (\ref{eq:ap1}) with $\sigma >1$ by using Kronecker's approximation theorem (see also \cite[Section 8.2]{Steuding1}). 

More than 50 years later from Bohr's paper \cite{Bohr}, Bagchi in his Ph.~D.~Thesis, proved that the Riemann hypothesis is true if and only if the Riemann zeta function can be approximated by itself in the sense of universality (see \cite{BagchiZ}). In order to state it, we need some notation. Let $D:= \{s \in {\mathbb{C}} : 1/2 < \Re (s) < 1\}$. By ${\rm{meas}} \{A\}$ we denote the Lebesgue measure of the set $A$, and, for $T>0$, we use the notation $\nu_T \{ \ldots \} := T^{-1} {\rm{meas}} \{ \tau \in [0,T] : \, \ldots \}$ where in place of dots some condition satisfied by $\tau$ is to be written. Then we have the following; The Riemann hypothesis is true if and only if, for any compact subset $K$ in the strip $D$ with connected complement and for any $\varepsilon> 0$, 
\begin{equation}
\liminf_{T \rightarrow \infty} \nu_T \Bigl\{ \max_{s \in K} \bigl| \zeta (s+{\rm{i}}\tau) - \zeta (s)\bigr| < \varepsilon\Bigr\} > 0.
\label{eq:seapprz}
\end{equation}
Note that it was shown in \cite{NaPaCorre}, the Riemann Hypothesis is also equivalent to the analogue of (\ref{eq:seapprz}) with $\zeta(s)$ replaced by the logarithm of the Riemann zeta function $\log \zeta (s)$.

Inspired by the work of Bagchi, Nakamura \cite{Nakamurasr1} showed the following property which might be called \textit{self-approximation} of the Riemann zeta function. For every algebraic irrational number $\beta \in\mathbb{R}$ or for almost all $\beta \in \mathbb{R}$, for every compact set $K\subset D$ with connected complement and every $\varepsilon>0$, one has
$$
\liminf_{T \rightarrow \infty} \nu_T \Bigl\{ \max_{s \in K} 
\bigl|\zeta(s+{\rm{i}}\tau)-\zeta(s+{\rm{i}}\beta\tau) \bigr| \Bigr\} < \varepsilon.
$$
Afterwards, Pa\'nkowski \cite{PankowskiRec} showed the self-approximation above for any irrational number $\beta$. Garunk\v{s}tis \cite{Garu} and Nakamura \cite{Nakamura3} investigated the self-approximation of the Riemann zeta function for non-zero rational numbers, independently. Unfortunately, the papers \cite{Garu} and \cite{Nakamura3} contain a gap in the proof of the main theorem, so actually their methods work only for $\log \zeta (s)$. The detail on this matter was presented in \cite{NaPaCorre}, where Nakamura and Pa\'nkowski partially filled this gap and prove the self-approximation of $\zeta(s)$ for $d=a/b$, where $a,b \in {\mathbb{N}}$ with $|a-b|\ne 1$ and $\gcd(a,b)=1$. Finally, Pa\'nkowski \cite{PanN16} proved the self-approximation of $\zeta(s)$ for every rational $d\ne 0, \pm 1$. Consequently, we have the following statement.
\begin{proposition}
For any $0 \ne \beta \in {\mathbb{R}}$, for any compact subset $K$ in the strip $D$ with connected complement and for any $\varepsilon> 0$, it holds that
\begin{equation}
\liminf_{T \rightarrow \infty} \nu_T \Bigl\{ \max_{s \in K} \bigl| \zeta (s+{\rm{i}}\tau) - 
\zeta (s+{\rm{i}}\beta\tau)\bigr| < \varepsilon\Bigr\} > 0.
\label{eq:seapplogrz}
\end{equation}
Furthermore, the Riemann zeta function in (\ref{eq:seapplogrz}) can be replaced by $\log \zeta$
\end{proposition}
Recall that we could prove the Riemann hypothesis if (\ref{eq:seapplogrz}) with $\beta=0$ would be true. The proposition above should be compared with probabilistic arguments for the Riemann hypothesis by Denjoy \cite{Denjoy} or Helson \cite{Helson} (see also \cite{Nakamurak}). 

Recently, Nakamura and Pa\'nkowski \cite{NaPaAu} investigated for which parameters $\lambda \in{\mathbb{C}}$ and $\beta \in {\mathbb{R}}$, the inequality 
\begin{equation}
\bigl|\zeta(s+\lambda+{\rm{i}}\beta \tau)-\zeta(s+{\rm{i}}\tau) \bigr| < \varepsilon \quad \mbox{or} \quad 
\bigl|\log \zeta(s+\lambda+{\rm{i}}\beta \tau)-\log \zeta(s+{\rm{i}}\tau) \bigr| < \varepsilon
\label{eq:zelogze}
\end{equation}
holds, assuming that $K$ is a compact set lying in the right half of the critical strip $D$ or in the half plane of absolute convergence $\sigma >1$.

\subsection{One dimensional case}
In this subsection, for simplicity, we consider almost periodicity and self-approximation of the following zeta or $L$-functions expressed as
\begin{equation}
{\mathcal{L}}(s)= \sum_{n=1}^\infty \frac{a(n)}{n^s} = 
\prod_p \prod_{k=1}^\eta \Bigl( 1-\frac{\alpha_k(p)}{p^s} \Bigr)^{-1} , \qquad \sigma>1,
\label{eq:ap1}
\end{equation}
where $\alpha_k(p) \in {\mathbb{C}}$, $|\alpha_k(p)|\le 1$, $1 \le k \le \varphi$. Recall $a(n)$ is defined by (\ref{eq:defsr3}). From the Euler product, the function ${\mathcal{L}}(s)$ does not vanish in the region of absolute convergence $\sigma >1$. On the other hand, the following lemma is well-known in probability theory. 
\begin{lemma}\label{lem:chawk}
Let $t_1,t_2 \in {\mathbb{R}}$, and $f(t)$ be a characteristic function on ${\mathbb{R}}$. Then we have
\begin{equation}
\bigl|f(t_1)-f(t_2)\bigr|^2 \le 4 \bigl| 1-f(t_1-t_2)\bigr|.
\label{eq:chara}
\end{equation}
\end{lemma}
By using this lemma, we have the next theorem.
\begin{theorem}\label{th:ap1}
Let $\sigma>1$ and $a(n) \ge 0$ for any $n \in {\mathbb{N}}$ in (\ref{eq:ap1}). Then we have 
\begin{equation}
\bigl|{\mathcal{L}}(\sigma+{\rm{i}}t_1)-{\mathcal{L}}(\sigma+{\rm{i}}t_2)\bigr|^2 \le 
4 \bigl| {\mathcal{L}}(\sigma)\bigr| 
\bigl| {\mathcal{L}}(\sigma) - {\mathcal{L}}(\sigma+{\rm{i}}(t_1-t_2))\bigr|.
\label{eq:ap2}
\end{equation}
\end{theorem}
\begin{proof}
The normalized function ${\mathcal{L}}(\sigma+{\rm{i}}t)/{\mathcal{L}}(\sigma)$ is a characteristic function on ${\mathbb{R}}$ by (II) of Theorem \ref{th:cla1} and the assumption $a(n) \ge 0$ for any $n \in {\mathbb{N}}$. Thus we can apply Lemma \ref{lem:chawk}. Hence it holds that
$$
\left| \frac{{\mathcal{L}}(\sigma+{\rm{i}}t_1)}{{\mathcal{L}}(\sigma)} - 
\frac{{\mathcal{L}}(\sigma+{\rm{i}}t_2)}{{\mathcal{L}}(\sigma)}\right|^2 \le 4 \left| 1 - 
\frac{{\mathcal{L}}(\sigma+{\rm{i}}(t_1-t_2))}{{\mathcal{L}}(\sigma)}\right|
$$
This inequality implies (\ref{eq:ap2}). 
\end{proof}

\begin{remark}
If $a(n) < 0$ for some $n \in {\mathbb{N}}$ in (\ref{eq:ap1}), the inequality (\ref{eq:ap2}) does not hold. For example, in the case of Dirichlet $L$-function $L(s)$ defined by (\ref{eq:defLds}), we have
$$
4 \bigl|L(\sigma)\bigr| \bigl| L(\sigma) - L(\sigma+{\rm{i}}(t_1-t_2))\bigr|
- \bigl|L(\sigma+{\rm{i}}t_1)-L(\sigma+{\rm{i}}t_2)\bigr|^2 =  -0.205831...
$$
for $\sigma=3/2$, $t_1=19.3$ and $t_2=82.9$ by Mathematica 8.0.
\end{remark}

The following corollary is a weaker version of \cite[Theorem 9.6]{Steuding1}. 
\begin{corollary}\label{cor:ap1}
Let $\sigma>1$ and $a(n) \ge 0$ for any $n \in {\mathbb{N}}$ in (\ref{eq:ap1}). Then for any $\varepsilon>0$, there exits $\tau \in {\mathbb{R}}$ such that
$$
\bigl| {\mathcal{L}} (\sigma + {\rm{i}}t + {\rm{i}}\tau) - {\mathcal{L}}(\sigma + {\rm{i}}t) \bigr| <\varepsilon 
\quad \mbox{for any } \,\, t \in {\mathbb{R}}.
$$
\end{corollary}
\begin{proof}
By Theorem \ref{th:ap1} and the assumption $a(n) \ge 0$, we only have to show that for any $\varepsilon'>0$, there exits $\tau \in {\mathbb{R}}$ such that $|{\mathcal{L}}(\sigma) - {\mathcal{L}}(\sigma+{\rm{i}}\tau)|<\varepsilon'$. This inequality is prove by the manner used in \cite[pp. 287--288]{KaVo} or the proof of Lemma \ref{lem:last} which will be appeared later. 
\end{proof}

The next corollary is a result similar to \cite[Theorem 3.7]{NaPaAu}. 
\begin{corollary}\label{cor:ap2}
Let $\sigma>1$ and $a(n) \ge 0$ for any $n \in {\mathbb{N}}$ in (\ref{eq:ap1}). Then for any $\lambda \in {\mathbb{R}}$, $1 \ne \beta \in {\mathbb{R}}$ and $\varepsilon>0$, there exits $t \in {\mathbb{R}}$ such that
$$
\bigl| {\mathcal{L}} (\sigma + {\rm{i}}\lambda + {\rm{i}}\beta t) - 
{\mathcal{L}}(\sigma + {\rm{i}}t) \bigr| < \varepsilon.
$$
\end{corollary}
\begin{proof}
By Theorem \ref{th:ap1} again, we only have to prove that for any $\varepsilon'>0$, there exits $t \in {\mathbb{R}}$ such that $|{\mathcal{L}}(\sigma) - {\mathcal{L}}(\sigma+{\rm{i}}\lambda + {\rm{i}}(\beta-1)\tau)|<\varepsilon'$. This inequality is also obtained by the manner used in \cite[pp. 287--288]{KaVo} or the proof of Lemma \ref{lem:last}. 
\end{proof}

When $\sum_{k=1}^\eta \alpha_k(p)^r \ge 0$ for any $r \in {\mathbb{N}}$ and $p \in {\mathbb{P}}$ in (\ref{eq:ap1}), we have the following theorem. We have to mention that this assumption is equivalent to the condition that the normalized function ${\mathcal{L}}(\sigma+it)/{\mathcal{L}}(\sigma)$ is a infinitely divisible characteristic function on ${\mathbb{R}}$ by (I) of Theorem \ref{th:cla1}.
\begin{theorem}\label{th:ap2}
Let $\sigma>1$ and $\sum_{k=1}^\eta \alpha_k(p)^r \ge 0$ for any $r \in {\mathbb{N}}$ and $p \in {\mathbb{P}}$ in (\ref{eq:ap1}). Then it holds that
\begin{equation}
\bigl|\log{\mathcal{L}}(\sigma+{\rm{i}}t_1)-\log{\mathcal{L}}(\sigma+{\rm{i}}t_2)\bigr|^2 \le 
4 \bigl| \log{\mathcal{L}}(\sigma)\bigr| 
\bigl| \log{\mathcal{L}}(\sigma) - \log{\mathcal{L}}(\sigma+{\rm{i}}(t_1-t_2))\bigr|.
\label{eq:ap3}
\end{equation}
\end{theorem}
\begin{proof}
From (\ref{eq:9.19}), we have 
\begin{equation*}
\begin{split}
&\log{\mathcal{L}} (s) = \sum_p \sum_{r=1}^{\infty} \sum_{k=1}^\eta \frac{1}{r} \alpha_{k}(p)^r p^{-rs} =
\sum_{n=1}^\infty \frac{A(n)}{n^s}, \\ & A(n) :=
\begin{cases}
 \sum_{k=1}^\eta \alpha_{k}(p)^r/r & n=p^r,\\
0 & \mbox{otherwise}.
\end{cases}
\end{split}
\end{equation*}
Hence the normalized function $\log{\mathcal{L}} (\sigma+{\rm{i}}t) /\log{\mathcal{L}} (\sigma)$ is a characteristic function by (\ref{eq:gdsd}) or Theorem \ref{th:d1}, and the assumption $\sum_{k=1}^\eta \alpha_k(p)^r \ge 0$. Therefore, we have (\ref{eq:ap3}) by modifying the proof of Theorem \ref{th:ap1}. 
\end{proof}

\begin{remark}
If $\sum_{k=1}^\eta \alpha_k(p)^r < 0$ for some $r \in {\mathbb{N}}$ or $p \in {\mathbb{P}}$ in (\ref{eq:ap1}), the inequality (\ref{eq:ap3}) does not hold. For example, in the case of Dirichlet $L$-function $L(s)$ defined by (\ref{eq:defLep}), 
\begin{equation*}
\begin{split}
&4 \bigl|\log L(\sigma)\bigr| \bigl| \log L(\sigma) - \log L(\sigma+{\rm{i}}(t_1-t_2))\bigr|
- \bigl| \log L(\sigma+{\rm{i}}t_1)- \log L(\sigma+{\rm{i}}t_2)\bigr|^2 \\ &= -0.16818... <0
\end{split}
\end{equation*}
for $\sigma=3/2$, $t_1=19.3$ and $t_2=82.9$ by Mathematica 8.0.

Next consider the function 
$$
Z_Q (s) := \frac{1}{(1-2^{-s})(1-{\rm{i}}2^{-s})(1+{\rm{i}}2^{-s})} = \frac{1}{(1-2^{-s})(1+2^{-2s})} 
= \frac{1+2^{-s}}{1-2^{-4s}}
$$
introduced in Example \ref{exa:q}. Then the graph of 
\begin{equation*}
\begin{split}
Q(t) := & \, 4|\log Z_Q(1/3)| \bigl|\log Z_Q(1/3) - \log Z_Q(1/3+7{\rm{i}}) \bigr| \\ &\quad 
- \bigl|\log Z_Q(1/3+{\rm{i}}t) - \log Z_Q(1/3+(t+7){\rm{i}})\bigr|^2
\end{split}
\end{equation*}
is the following.
\begin{figure}[h]
\includegraphics[height=5cm, width=16cm]{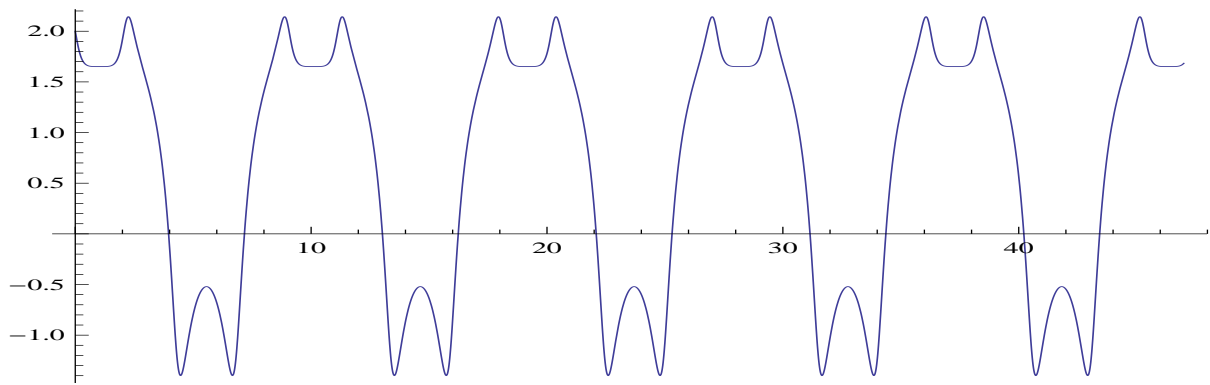}
\caption{$\{ Q(t) : 0 \le t \le 47 \}$}
\label{f1}
\end{figure}
Hence $Z_Q(1/3+{\rm{i}}t)$ does not satisfy (\ref{eq:ap3}). On the contrary, $Z_Q(1/3+{\rm{i}}t)$ fulfills (\ref{eq:ap2}) from  Example \ref{exa:q}. 
\end{remark}

By the manners similar to the proofs of Corollaries \ref{cor:ap1} and \ref{cor:ap2}, we have the following. Note that Corollary \ref{cor:ap3} is also a weaker version of \cite[Theorem 9.6]{Steuding1}. 
\begin{corollary}\label{cor:ap3}
Let $\sigma>1$ and $\sum_{k=1}^\eta \alpha_k(p)^r \ge 0$ for any $r \in {\mathbb{N}}$ and $p \in {\mathbb{P}}$ in (\ref{eq:ap1}). Then for any $\varepsilon>0$, there exits $\tau \in {\mathbb{R}}$ such that
$$
\bigl| \log {\mathcal{L}} (\sigma + {\rm{i}}t + {\rm{i}}\tau) - \log {\mathcal{L}}(\sigma + {\rm{i}}t) \bigr| 
<\varepsilon \quad \mbox{for any } \,\, t \in {\mathbb{R}}.
$$
\end{corollary}
\begin{corollary}\label{cor:ap4}
Let $\sigma>1$ and $\sum_{k=1}^\eta \alpha_k(p)^r \ge 0$ for any $r \in {\mathbb{N}}$ and $p \in {\mathbb{P}}$ in (\ref{eq:ap1}). Then for any $\lambda \in {\mathbb{R}}$, $1 \ne \beta \in {\mathbb{R}}$ and $\varepsilon>0$, there exits $t \in {\mathbb{R}}$ such that
$$
\bigl| \log {\mathcal{L}} (\sigma + {\rm{i}}\lambda + {\rm{i}}\beta t) - 
\log {\mathcal{L}}(\sigma + {\rm{i}}t) \bigr| < \varepsilon.
$$
\end{corollary}

\subsection{Multidimensional case}
In this subsection, we consider the multidimensional polynomial Euler product (\ref{eq:defpe4}) with $\min_{1\le l\le \varphi} \Re \langle \vc_l,\vs\rangle >1$, $\alpha_{lk}(p) \in {\mathbb{C}}$, $|\alpha_{lk}(p)|\le 1$, $1 \le k \le \eta$, $1 \le l \le \varphi$, and the condition (A1) or (A2) introduced in Section 4. The following is the multidimensional version of Lemma \ref{lem:chawk}. 
\begin{lemma}\label{lem:chawk2}
Let $\vt_1,\vt_2 \in {\mathbb{R}}^d$, and $f(\vt)$ be a characteristic function on ${\mathbb{R}}^d$. Then we have
\begin{equation}
\bigl|f(\vt_1)-f(\vt_2)\bigr|^2 \le 4 \bigl| 1-f(\vt_1-\vt_2)\bigr|.
\label{eq:chara2}
\end{equation}
\end{lemma}
By using the lemma above and Theorem \ref{th:clam1}, and modifying the proof of Theorem \ref{th:ap1}, we have the following theorem.
\begin{theorem}\label{th:apmd1}
Let $\prod_{l=1}^\varphi a_l(n_l) \ge 0$ for all $(n_1, \ldots , n_\varphi) \in {\mathbb{N}}^\varphi$ in (\ref{eq:defsz4}). Then we have 
\begin{equation}
\bigl| Z_E(\vsig+{\rm{i}}\vt_1) - Z_E(\vsig+{\rm{i}}\vt_2)\bigr|^2 \le 
4 \bigl| Z_E(\vsig)\bigr| \bigl| Z_E(\vsig) - Z_E(\vsig+{\rm{i}}(\vt_1-\vt_2))\bigr|.
\label{eq:apmd2}
\end{equation}
\end{theorem}

\begin{proof}
The normalized function $Z_E(\vsig+{\rm{i}}\vt) / Z_E(\vsig)$ is a characteristic function from the assumption $\prod_{l=1}^\varphi a_l(n_l) \ge 0$ for all $(n_1, \ldots , n_\varphi) \in {\mathbb{N}}^\varphi$ and Theorem \ref{th:dm1}. Therefore, we can apply Lemma \ref{lem:chawk2}.
\end{proof}
We have the next corollaries by Theorem \ref{th:apmd1}, Lemma \ref{lem:last} and the manner used in the proofs of Corollaries \ref{cor:ap1} and \ref{cor:ap2}.

\begin{corollary}\label{cor:apmd1}
Let $\prod_{l=1}^\varphi a_l(n_l) \ge 0$ for all $(n_1, \ldots , n_\varphi) \in {\mathbb{N}}^\varphi$. Then for any $\varepsilon>0$, there exits $\Vec\tau \in {\mathbb{R}}$ such that
$$
\bigl| Z_E (\vsig + {\rm{i}} \vt + {\rm{i}} \Vec\tau) - Z_E (\vsig + {\rm{i}}\vt) \bigr| <\varepsilon 
\quad \mbox{for any } \,\, \vt \in {\mathbb{R}}^d.
$$
\end{corollary}
\begin{corollary}\label{cor:apmd2}
Let $\prod_{l=1}^\varphi a_l(n_l) \ge 0$ for all $(n_1, \ldots , n_\varphi) \in {\mathbb{N}}^\varphi$. Then for any $\Vec \lambda \in {\mathbb{R}}^d$, $1 \ne \beta \in {\mathbb{R}}$ and $\varepsilon>0$, there exits $\vt \in {\mathbb{R}}^d$ such that
$$
\bigl| Z_E (\vsig + {\rm{i}} \Vec\lambda + {\rm{i}} \beta \vt) - Z_E (\vsig + {\rm{i}} \vt) \bigr| < \varepsilon.
$$
\end{corollary}

Similarly, we have the following theorem.
\begin{theorem}\label{th:apmd2}
Let $\sum_{k=1}^\eta \alpha_{lk}(p)^r \ge 0$ for all $r \in {\mathbb{N}}$, $p\in\Prime$ and $1 \le l \le \varphi$ in (\ref{eq:defpe4}). Then it holds that
\begin{equation}
\bigl| \log Z_E (\vsig+{\rm{i}}\vt_1) - \log Z_E (\vsig+{\rm{i}}\vt_2) \bigr|^2 \le 
4 \bigl| \log Z_E (\vsig)\bigr| 
\bigl| \log Z_E (\vsig) - \log Z_E (\vsig+{\rm{i}}(\vt_1-\vt_2))\bigr|.
\label{eq:mdap3}
\end{equation}
\end{theorem}
\begin{proof}
In the view of (\ref{eq:9.19}), we have 
\begin{equation*}
\begin{split}
&\log Z_E(\vs) = 
\sum_p \sum_{r=1}^{\infty} \sum_{l=1}^\varphi \sum_{k=1}^\eta \frac{1}{r} 
\alpha_{lk}(p)^r p^{-r\langle\vc_l,\vs \rangle} =
\sum_{l=1}^\varphi \sum_{n_l=1}^\infty \frac{A_l(n_l)}{n_l^s}, \\ &A_l(n_l) :=
\begin{cases}
 \sum_{k=1}^\eta \alpha_{lk}(p)^r /r & n_l=p^r,\\
0 & \mbox{otherwise}.
\end{cases}
\end{split}
\end{equation*}
Thus the normalized function $\log Z_E (\vsig+{\rm{i}}\vt)/ \log Z_E (\vsig)$ is a characteristic function by Theorem \ref{th:dm1} and the condition $\sum_{k=1}^\eta \alpha_{lk}(p)^r \ge 0$. Hence we can apply Lemma \ref{lem:chawk2}. 
\end{proof}

By using Theorem \ref{th:apmd2} and Lemma \ref{lem:last}, we obtain the following corollaries which are multidimensional cases of Corollaries \ref{cor:ap3} and \ref{cor:ap4}. 
\begin{corollary}\label{cor:apmd3}
Let $\sum_{k=1}^\eta \alpha_{lk}(p)^r \ge 0$ for all $r \in {\mathbb{N}}$, $p\in\Prime$ and $1 \le l \le \varphi$. Then for any $\varepsilon>0$, there exits $\Vec\tau \in {\mathbb{R}}$ such that
$$
\bigl| \log Z_E (\vsig + {\rm{i}} \vt + {\rm{i}} \Vec\tau) - \log Z_E (\vsig + {\rm{i}}\vt) \bigr| <\varepsilon 
\quad \mbox{for any } \,\, \vt \in {\mathbb{R}}^d.
$$
\end{corollary}
\begin{corollary}\label{cor:apmd4}
Let $\sum_{k=1}^\eta \alpha_{lk}(p)^r \ge 0$ for all $r \in {\mathbb{N}}$, $p\in\Prime$ and $1 \le l \le \varphi$. Then for any $\Vec \lambda \in {\mathbb{R}}^d$, $1 \ne \beta \in {\mathbb{R}}$ and $\varepsilon>0$, there exits $\vt \in {\mathbb{R}}^d$ such that
$$
\bigl| \log Z_E (\vsig + {\rm{i}} \Vec\lambda + {\rm{i}} \beta \vt) - \log Z_E (\vsig + {\rm{i}} \vt) \bigr| 
< \varepsilon.
$$
\end{corollary}

At the end of this paper, we show the remained lemma used many times in the proofs of corollaries in this section.
\begin{lemma}\label{lem:last}
For any $\varepsilon>0$, there exits $\Vec \tau \in {\mathbb{R}}^d$ such that
$$
\bigl| Z_E (\vsig) - Z_E (\vsig + {\rm{i}} \Vec \tau) \bigr| < \varepsilon \quad \mbox{and} \quad
\bigl| \log Z_E (\vsig) - \log Z_E (\vsig + {\rm{i}} \Vec \tau) \bigr| < \varepsilon.
$$
\end{lemma}
In order to prove this, we quote the following two propositions.
\begin{proposition}[see {\cite[Theorem 7.9]{Apo2}}]\label{pro:kroap1}
If $\phi_1 , \ldots ,\phi_n$ are arbitrary real numbers, if real numbers $\theta_1,\ldots,\theta_n$ are linearly independent over the rationals, and if $\varepsilon>0$ is arbitrary, then there exists a real number $t$ and integers $h_1, \ldots, h_n$ such that
$$
|t\theta_k - h_k - \phi_k| < \varepsilon, \qquad 1 \le k \le n. 
$$
\end{proposition}
\begin{proposition}[see {\cite[Proposition 2.2]{Nakamurasr1}}]
Let $p_n$ be the $n$-th prime number and $\omega_1, \omega_2, \ldots ,$ $\omega_m$ with $\omega_1 =1$ be algebraic real numbers which are linearly independent over the rationals. Then $\{ \log p_n ^{\omega_l}\}_{n \in {\mathbb{N}}}^{1 \le l \le m}$ is also linearly independent over the rationals. 
\label{pro:abaker}
\end{proposition}

\begin{proof}[Proof of Lemma \ref{lem:last}]
Obviously, we have
$$
Z_E (\vsig) - Z_E (\vsig + {\rm{i}} \Vec \tau) = Z_E (\vsig) 
\bigl( 1 - \exp \bigl(\log Z_E (\vsig) - \log Z_E (\vsig + {\rm{i}} \Vec \tau) \bigr) \bigr).
$$
Hence we only have to prove that there exits $\Vec \tau \in {\mathbb{R}}^d$ such that for any $\varepsilon>0$, $|\log Z_E (\vsig) - \log Z_E (\vsig + {\rm{i}} \Vec \tau) | < \varepsilon$. 

Now we suppose that a pair of vectors $\vc_{1}, \ldots , \vc_\varphi$ satisfies (A2). By Propositions \ref{pro:kroap1} and \ref{pro:abaker}, for any $\delta>0$, there exists $\Vec \tau_0\in\rd$ such that
$$
|p^{-{\rm{i}} \langle \vc_l, \Vec \tau_0 \rangle} -1| < \delta, \qquad 2 < p \le M, \quad 1 \le l \le \varphi.
$$
Obviously, we have $|x^m-1| = |x-1| |x^{m-1}+\cdots+ 1| \le m |x-1|$, for any $m \in {\mathbb{N}}$ and $|x| \le 1$. Thus it holds that
$$
|p^{-r {\rm{i}} \langle \vc_l, \Vec \tau_0 \rangle} -1| < m \delta, \qquad 2 < p ,r \le m, \quad 1 \le l \le \varphi
$$
when $\Vec \tau_0\in\rd$ satisfies $|p^{-{\rm{i}} \omega_l \tau_0} -1| < \delta$. 
In the view of (\ref{eq:9.19}), one has
$$
\log Z_E(\vs) = \sum_p \sum_{r=1}^{\infty} \sum_{l=1}^\varphi \sum_{k=1}^\eta \frac{1}{r} 
\alpha_{lk}(p)^r p^{-r\langle\vc_l,\vs \rangle} .
$$
By using the assumption $|\alpha_{lk}(p)| \le 1$ and $\min_{1\le l\le \varphi} \Re \langle \vc_l,\vs\rangle >1$ and the absolute convergence of the series $\sum_p \sum_{r=1}^{\infty} \sum_{l=1}^\varphi \sum_{k=1}^\eta r^{-1} \alpha_{lk}(p)^r p^{-r\langle\vc_l,\vs \rangle}$, we can find a natural number $M$ such that
$$
\sum_{p,r>M} \!\!{}' \sum_{l=1}^\varphi \sum_{k=1}^\eta \frac{1}{r} 
|\alpha_{lk}(p)|^r p^{-r\langle\vc_l,\vsig \rangle} \le \sum_{p,r>M} \!\!{}' \sum_{l=1}^\varphi 
\frac{\eta}{r} p^{-r\langle\vc_l,\vsig \rangle} < \varepsilon,
$$
where $\sum_{p,r>M}'$ be a sum taken over $p>M$ or $r>M$. Hence we have
\begin{equation*}
\begin{split}
&\bigl|\log Z_E(\vsig)-\log Z_E(\vsig+\Vec \tau_0) \bigr| \le 
\sum_p \sum_{r=1}^{\infty} \sum_{l=1}^\varphi \sum_{k=1}^\eta \frac{1}{r} |\alpha_{lk}(p)|^r 
p^{-r\langle\vc_l,\vsig \rangle} \bigl| p^{-r {\rm{i}} \langle \vc_l, \Vec \tau_0 \rangle} -1 \bigr|\\
\le & \sum_p^M \sum_{r=1}^M \sum_{l=1}^\varphi \frac{\eta}{r} p^{-r\langle\vc_l,\vsig \rangle} 
\bigl| p^{-r {\rm{i}} \langle \vc_l, \Vec \tau_0 \rangle} -1 \bigr| + \sum_{p,r>M} \!\!{}' \sum_{l=1}^\varphi 
\frac{\eta}{r} p^{-r\langle\vc_l,\vsig \rangle} 
\bigl| p^{-r {\rm{i}} \langle \vc_l, \Vec \tau_0 \rangle} -1 \bigr|\\
< &\, \eta M^3 \delta + 2 \varepsilon.
\end{split}
\end{equation*}
Therefore, we have $|\log Z_E (\vsig) - \log Z_E (\vsig + {\rm{i}} \Vec \tau_0) | < 3\varepsilon$ by taking $\delta:= \eta^{-1} M^{-3} \varepsilon$.

Next suppose that a pair of vectors $\vc_{1}, \ldots , \vc_\varphi$ satisfies the condition (A1). Then there exits $\vt_0\in\rd$ such that $(\langle \vc_1, \vt_0 \rangle, \ldots , \langle \vc_\varphi, \vt_0 \rangle) =(\omega_1, \omega_2, \ldots ,\omega_\varphi)$, where $\omega_1,\omega_2, \ldots ,\omega_\varphi$ with $\omega_1=1$ be algebraic real numbers which are linearly independent over the rationals. Thus we can show the case (A1) by modifying the proof of the case (A2). 
\end{proof}

\subsection*{Acknowledgments}
The author was partially supported by JSPS grant 16K05077 and Japan-France Research Cooperative Program (JSPS and CNRS). 

 

\begin{thebibliography}{1}
\bibitem{AN12q}
{\rm T.~Aoyama \and T.~Nakamura}, `Behaviors of multivariable finite Euler products in probabilistic view', {\it{Math.~Nachr.}} {\bf{286}} (2013), no.~17-18, 1691--1700. {http://arxiv.org/abs/1204.4043}.

\bibitem{AN12s}
{\rm T.~Aoyama \and T.~Nakamura}, `Multidimensional Shintani zeta functions and zeta distributions on $\rd$', {\it{Tokyo Journal Mathematics}} {\bf{36}} (2013), no.~2, 521--538. {http://arxiv.org/abs/1204.4042}.

\bibitem{ANPE}
{\rm T.~Aoyama \and T.~Nakamura}, `Multidimensional polynomial Euler products and infinitely divisible distributions on $\rd$', submitted (2012) {http://arxiv.org/abs/1204.4041}.

\bibitem{AY}
{\rm T.~Aoyama \and K.~Yoshikawa}, `Representations of discrete distributions on $\rd$ by multiple zeta functions', {\it{preprint}}.

\bibitem{Apo}
{\rm T.~M.~Apostol}, 
{\em Introduction to Analytic Number Theory} (Undergraduate Texts in Mathematics, Springer, 1976). 

\bibitem{Apo2}
{\rm T.~M.~Apostol}, {\em Modular functions and Dirichlet series in Number Theory } (Graduate Texts in Mathematics 41, Springer, 1990). 

\bibitem{BagchiZ}
B.~Bagchi, `A joint universality theorem for Dirichlet $L$-functions', {\it{Math.~Z.}} {\bf{181}} (1982), no.~3, 319--334.

\bibitem{Baker}
{\rm A.~Baker}, {\em Transcendental number theory } (Cambridge Mathematical Library, Cambridge University Press, Cambridge, 1975). 

\bibitem{Bohr}
H.~Bohr, `Uber eine quasi-periodische Eigenschaft Dirichletscher Reihen mit Anwendung auf die Dirichletschen $L$-Funktionen', (German), {\it{Math.~Ann.}} {\bf{85}} (1922), no.~1, 115--122. 

\bibitem{Cohen} 
{\rm H.~Cohen}, {\em Number theory. Vol.~II. Analytic and modern tools} (Graduate Texts in Mathematics, 240. Springer, New York, 2007). 

\bibitem{Denjoy}
A.~Denjoy, `L'Hypoth\'{e}se de Riemann sur la distribution des z\'{e}ros de $\zeta (s)$, reli\'{e}e
\`{a} la th\'{e}orie des probabilit\'{e}s', {\it{Comptes Rendus Acad.~Sci.~Paris}} {\bf{192}} (1931), 656--658.

\bibitem{Garu} R.~Garunk\v{s}tis, `Self-approximation of Dirichlet $L$-functions', {\it{J.~Number Theory}} \textbf{131} (2011), no.~7, 1286--1295.

\bibitem{GiSt} 
E.~Girondo and J.~Steuding, `Effective estimates for the distribution of values of Euler products', {\it{Monatsh.~Math.}} {\bf{145}} (2005), no. 2, 97--106. 

\bibitem{GK68}
{\rm B.~V.~Gnedenko \and A.~N.~Kolmogorov}, {\em Limit Distributions for Sums of Independent Random Variables (Translated from the Russian by Kai Lai Chung)}, (Addison-Wesley, 1968).

\bibitem{Helson}
H.~Helson, `Compact groups and Dirichlet series', {\it{Arkiv f\"or Matematik}} {\bf{8}} (1969), 139--143.

\bibitem{KaVo}
A.~A.~Karatsuba and S.~M.~Voronin {\it{The Riemann zeta-function}}. Translated from the Russian by Neal Koblitz. de Gruyter Expositions in Mathematics, 5. Walter de Gruyter $\&$ Co., Berlin, 1992. 

\bibitem{Khi}
{\rm A.~Ya.~Khinchine}, {\em Limit Theorems for Sums of Independent Random Variables (in Russian)}, (Moscow and Leningrad, 1938).

\bibitem{Lau} A.~Laurin\v cikas, {\it{Limit theorems for the Riemann zeta-function}}, Mathematics and its Applications, 352. Kluwer Academic Publishers Group, Dordrecht, 1996.

\bibitem{Lin}
{\rm G.~D.~Lin \and C.-~Y.~Hu}, `The Riemann zeta distribution', {\em Bernoulli} {\bf{7}} (2001) 817--828.

\bibitem{LSm}
{\rm A.~Lindner \and K.~Sato}, `Properties of stationary distributions of a sequence of generalized Ornstein-Uhlenbeck processes', {\em Math.~Nachr.} {\bf{284}} (2011) 2225--2248.

\bibitem{LO}
{\rm J.~V.~Linnik \and I.~V.~Ostrovskii}, `Decompositions of Random Variables and Vectors', (American Mathematical Society, Providence, RI, 1977 (Translation from the Russian original of 1972)).

\bibitem{Mas} K.~Matsumoto, `Probabilistic value-distribution theory of zeta-functions', {\it{Sugaku Expositions}} {\bf{17}} (2004), no.~1, 51--71. 

\bibitem{Nakamurasr1} 
{\rm T.~Nakamura}, `The joint universality and the generalized strong recurrence for Dirichlet $L$-functions', {\em Acta Arith.} {\bf{138}} no.~4 (2009) 357--362. 

\bibitem{Nakamura3} T.~Nakamura, `The generalized strong recurrence for non-zero rational parameters', {\it{Archiv der Mathematik}} {\bf{95}} (2010), 549--555.

\bibitem{Nakamurak} T.~Nakamura, `The generalized strong recurrence and the Riemann Hypothesis', Functions in Number Theory and Their Probabilistic Aspects, {\it{RIMS K\^oky\^uroku Bessatsu}}, {\bf{B34}}, (2012), 265--276.

\bibitem{Nakamura12} 
{\rm T.~Nakamura}, `A quasi-infinitely divisible characteristic function and its exponentiation', {\it{Statist.~Probab.~Lett.}} {\bf{83}} (2013), no.~10, 2256--2259. 

\bibitem{NaPaCorre}
T.~Nakamura and \L.~Pa\'{n}kowski, `Erratum to: The generalized strong recurrence for non-zero rational parameters', {\it{Arch.~Math. (Basel)}} {\bf{99}} (2012), no.~1, 43--47. 

\bibitem{NaPaAu}
T.~Nakamura and \L.~Pa\'{n}kowski, `Self-approximation for the Riemann zeta function', {\it{Bulletin of the Australian Mathematical Society}} {\bf{87}} (2013), 452--461. 

\bibitem{PankowskiRec} \L.~Pa\'{n}kowski, `Some remarks on the generalized strong recurrence for L-functions', New directions in value-distribution theory of zeta and $L$-functions, 305--315, Ber.~Math., Shaker Verlag, Aachen, 2009.

\bibitem{PanN16} \L.~Pa\'{n}kowski, `Joint universality and generalized strong recurrence with rational parameter', {\it{J.~Number Theory}} {\bf{163}} (2016), 61--74. 

\bibitem{Rudin} W.~Rudin, {\it{Real and complex analysis}}. Third edition. McGraw-Hill Book Co., New York, 1987. 

\bibitem{S99} 
{\rm K.~Sato}, {\em L\'evy Processes and Infinitely Divisible Distributions}, Cambridge Studies in Advanced Mathematics, {\bf{68}} (Cambridge University Press, Cambridge, 1999). 


\bibitem{Sato12} 
{\rm K.~Sato}, `Stochastic integrals with respect to Levy processes and infinitely divisible distributions', {\it{Sugaku Expositions}} {\bf{27}} (2014), no.~1, 19--42. http://ksato.jp/.

\bibitem{Steuding1} 
{\rm J.~Steuding}, {\em Value-Distribution of L-functions}, Lecture Notes in Mathematics, 1877, Springer, Berlin, 2007.

\bibitem{Tit}
{\rm E.~C.~Titchmarsh}, {\em The theory of the Riemann zeta-function. Second edition. Edited and with a preface by D.~R.~Heath-Brown} (The Clarendon Press, Oxford University Press, New York, 1986). 

\end{thebibliography}
\end{document}